\documentclass[10pt]{article}
\usepackage{latexsym,amsfonts,amsthm,amsmath,amscd,amssymb,color}
\usepackage{makeidx}
\usepackage{todonotes}
\usepackage{mathabx}
\usepackage{float}
\usepackage[colorlinks=true,citecolor=blue,linkcolor=black,]{hyperref}
\usepackage[all]{xy}
\renewcommand{\baselinestretch}{1.2}
\usepackage{geometry}
\usepackage{cancel}

\theoremstyle{plain}
\newtheorem{lemma}{Lemma}[section]
\newtheorem{theorem}[lemma]{Theorem}
\newtheorem{proposition}[lemma]{Proposition}
\newtheorem{corollary}[lemma]{Corollary}

\theoremstyle{definition}
\newtheorem{definition}[lemma]{Definition}
\newtheorem{example}[lemma]{Example}
\newtheorem{remark}[lemma]{Remark}
\newtheorem*{definition*}{Definition}

\theoremstyle{remark}

\newtheorem{convention}[lemma]{Convention}

\newcommand{\dd}{\mathrm{d}}
\newcommand{\X}{\mathfrak {X} }
\newcommand{\E}{\mathcal{E}}
\newcommand{\lb}{\left[ \cdot\,,\cdot\right] }
\usepackage[pagewise]{lineno}%\linenumbers
\usepackage{cite}

\usepackage{titling}

\title{A series of Nash resolutions of a singular foliation}

\author{Ruben Louis\\  \href{ruben.louis@mathematik.uni-goettingen.de}{ruben.louis@mathematik.uni-goettingen.de}\thanks{Department of Mathematics, Jilin University, Changchun 130012, Jilin, China.}~\thanks{Institut f\"ur Mathematik, Georg-August-Universit\"at G\"ottingen, G\"ottingen, Germany.}}

\date{}

\begin{document}

\maketitle

\begin{abstract}
We construct a series of blowups $(\widetilde M_i,\pi_i)_{i\in \mathbb N_0}$ of a singular foliation by applying to the universal Lie $\infty$-algebroid of a singular foliation  %monoidal transformations of 
the so-called Nash modification.
%of an affine variety  and a resolution method due  to 
For $i=0$, we recover a blowup introduced Sinan Sertöz, and for $i=1$,  we recover a notion due to Omar Mohsen.  %\cite{MohsenOmar}. 
One of the important features is that any singular foliation becomes a Debord foliation (= projective singular foliation) after one blowup.
Examples are also given.
\end{abstract}

\scalebox{0.8}{\textbf{Keywords}: Singular foliations, singularities, Nash blowup, homotopy Lie algebras.}
\tableofcontents

\section*{Introduction}
Singular foliations generalize the notion of regular foliations by allowing leaves of different dimensions. They arise frequently in differential or algebraic geometry. Here, as in \cite{LLL}, we unify \cite{Hermann,AndroulidakisIakovos,AndroulidakisZambon,Cerveau, Debord,LLS} in smooth differential geometry and \cite{zbMATH03423310,Ali-Sinan} in holomorphic differential geometry by defining a singular foliation on a smooth, complex, algebraic, or real analytic manifold $M$,  with sheaf of functions $\mathcal O$, to be a subsheaf $\mathfrak F \colon U\longrightarrow\mathfrak F(U )$ of the sheaf of vector fields $\mathfrak X$, which is closed under the Lie bracket and locally finitely generated as an $\mathcal O$-module. 
By Hermann's theorem \cite{Hermann}, this is enough to induce a partition of the manifold $M$ into {immersed} submanifolds of possibly different dimensions, called \emph{leaves} of the singular foliation. Singular foliations appear for instance as orbits of Lie group actions, with possibly different dimensions.  % or as symplectic leaves of a Poisson structure. 
In the realm of Poisson Geometry, we encounter a particularly intricate class of singular foliations known as “the symplectic leaves of a Poisson structure”, see \cite{LPV,CFM}. When all the leaves have the same dimension, we recover the usual \textquotedblleft regular foliations\textquotedblright \cite{DiederichKlas, LLL}. We refer to \cite{LLL} Section I.4 for a detailed list of examples.

In this paper, we address  blowups of a singular foliation  $\mathfrak F  $ on $M$, i.e.,  pairs $(\widetilde M, \pi)$ such that
\begin{enumerate}
\item $\pi\colon \widetilde M \to M$ is onto and proper;
\item the restriction $\pi|_{M_{\mathrm{reg}, \mathfrak F}} \colon \pi^{-1}(M_{\mathrm{reg}, \mathfrak F}) \to M_{\mathrm{reg}, \mathfrak F}$ to the regular points $M_{\mathrm{reg}, \mathfrak F}$ of $\mathfrak F$ is one-to-one,
\item
 the pullback $\pi^! \mathfrak F$ of $\mathfrak F$ on $\widetilde M$ exists and satisfies $\pi^! \mathfrak F |_{\pi^{-1}(M_{\mathrm{reg}, \mathfrak F})} \simeq \mathfrak F |_{M_{\mathrm{reg},\mathfrak F}}$. 
 \end{enumerate}
 This notion has been studied in various forms by many authors \cite{NistorVictor2019DoLg,DebordClaire2017BcfL, MohsenOmar}.  We %enlarge the construction of O. Mohsen \cite{MohsenOmar} for
consider singular foliations, which admit (locally) geometric resolutions, i.e., those for which there exists an anchored complex of vector bundles 

$$ (E,\dd,
\rho) \colon  \xymatrix{  \ar[r] & E_{-i-1} \ar[r]^{{\dd^{(i+1)}}} \ar[d] & 
     E_{ -i} \ar[r]^{{\dd^{(i)}}}
     \ar[d] & E_{-i+1} \ar[r] \ar[d] & \ar@{..}[r] & \ar[r]^{{\dd^{(2)}}}& E_{-1} \ar[r]^{\rho} \ar[d]& TM \ar[d] \\ 
      \ar@{=}[r] & \ar@{=}[r] M  &  \ar@{=}[r] M 
      &  \ar@{=}[r] M  &\ar@{..}[r] & \ar@{=}[r]   &  \ar@{=}[r] M  & M}
    $$
    such that the following complex of sheaves      \begin{equation}
           \label{}{\longrightarrow} \Gamma({E_{ -i-1}})
     \stackrel{\dd^{(i+1)}}{\longrightarrow} \Gamma({E_{-i}})
     \stackrel{\dd^{(i)}}{\longrightarrow}{\Gamma(E_{-i+1})}{\longrightarrow}\cdots  {\longrightarrow} \Gamma(E_{-1})  
     \stackrel{\rho}{\longrightarrow} \mathfrak F.
      \end{equation}  is exact. %We provide a new approach to resolutions of singularities for these singular foliations by using Nash blowup techniques.
      Those singular foliations were introduced and studied by Laurent-Gengoux, Lavau, and Strobl \cite{LLS}.  It is a quite natural to work with this class of singular foliation, as it contains the class of (locally) real analytic singular foliations. It is also a natural object in the holomorphic setting, since $\mathfrak F$ is then a coherent sheaf and such geometric resolutions always exist locally.  
      
      We will use geometric resolutions to construct a  sequence of blowups $\left(\widetilde M_i,\pi_i\right)_{i\in \mathbb N_0}$.
There is a very long story behind our construction. Let us be precise on the relations with other works : we were first inspired by Omar Mohsen \cite{MohsenOmar}, who introduced a notion of  blowup of a smooth manifold along the singular leaves of a singular foliation which does not consist of blowing up along a singular leaf as in \cite{Wang-Tang} or \cite{DebordClaire2017BcfL} or gluing Lie groupoids as in \cite{NistorVictor2019DoLg}. The construction of Mohsen extends an older idea that consists in  replacing every singular point of a singular foliation by the limiting positions of the tangent spaces of the nearly regular leaves. This  method goes back to the mathematician J. Nash \cite{Nobile}, and is mainly used in algebraic geometry for desingularisation of affine varieties or schemes. To the best of my knowledge,  Sinan Sertöz \cite{Ali-Sinan} was the first to apply this method in his Ph.D. dissertation  to compute the Baum-Bott residues of singular holomorphic foliations. A more general construction fo generic coherent sheaves was done earlier by \cite{Rossi}. For further details, see also \cite{JPTS}. In fact, Sinan Sertöz went further by applying the Nash construction to coherent subsheaves of locally free sheaves, thereby generalizing the work of Nobile \cite{Nobile}.

In the complex setting, therefore, our blowups $\left(\widetilde M_i,\pi_i\right)$  coincide for $i=0,1$ with blowup of $ \mathfrak F$, seen as a holomorphic coherent sheaf as in \cite{Rossi,Ali-Sinan, JPTS}. Also, in the smooth setting, we will see that $\widetilde M_1$ is the blowup space of the singular foliation  $(M, \mathfrak F)$  defined by of O. Mohsen \cite{MohsenOmar}. 
    In general, the blowup spaces  $(\widetilde M_i)_{i\geq 0}$ are  Nash blowups (also called Nash modification) of coherent sheaves. This coincidence has practical consequences: for instance, the smoothness of the $\widetilde M_i$'s can be studied using classical results on Nash modifications, as in \cite{Ali-Sinan}.

  Now,  when dealing with a singular foliation, rather than just a coherent sheaf or a sheaf that admits a geometric resolution, additional structures emerge. Specifically, this context gives rise to Lie algebroids or Lie $n$-algebroid structures. For instance,
    we will see that $\pi_1^!\mathfrak F$ is  always a Debord singular foliation on $ \widetilde{M}_1$, meaning it is the image of a Lie algebroid whose anchor map is injective on an open dense subset. 
It is shown in \cite{LLS,CLRL} that “behind” any singular foliation admitting a geometric resolution there is a Lie $\infty$-algebroid constructed over a geometric resolution $(E,\dd,\rho)$ of $\mathfrak{F}$, which is unique up to homotopy. The latter is referred to as an universal Lie $\infty$-algebroid of $\mathfrak F$. 
%It is quite natural to explore the extent to which this structure can be used to desingularize the singular foliation. This article continues the approach of using the universal Lie $\infty$-algebroid of $\mathfrak F$ as a tool to draw conclusions about singular foliation $\mathfrak F$.
For instance, it was employed by S. Lavau in \cite{SL} to define the modular class of a singular foliation or \cite{RL_symmetries} to study symmetries of singular foliations, see also \cite{CLG-Ryv, Karandeep} for other applications.
%Our approach uses the universal Lie $\infty$-algebroid of a singular foliation to unify various methods of desingularizing singular foliations, while also incorporating additional structures. 

The existence of a structure of a universal Lie $\infty$-algebroid on geometric resolutions has consequences for Nash modifications. In fact, this structure is necessary to check that the blowups $(\widetilde{M}_i,\pi_i) $ satisfy the third condition in a definition of a blowup, namely that the pull-back of $\mathfrak F$ exists and is a singular foliation.
More precisely, we do not apply the Nash modification idea  directly on the tangent space of our singular foliation $\mathfrak{F}\subseteq \mathfrak X(M)$ but on the images of the differential maps $\dd^{(i+1)}\colon E_{-i-1}\to E_{-i},\; i\geq 1$ and of the anchor map $\rho\colon E_{-1}\rightarrow TM$ of a geometric resolution $(E,\dd,\rho)$. 
 The choice of considering the images of all the $\dd^{(i+1)}$'s allows generalizing the Nash construction to the whole universal  Lie $\infty$-algebroid of the singular foliation built over a geometric resolution. As we said, we then recover several notions of resolution of singularities for $i=0,1$. 
 %Indeed, Mohsen's construction is the $i=1$ type of a series $\pi_i\colon\widetilde{M}_i\to M$ indexed by $i\in \mathbb N_0$ of blowups of a singular foliation associated to the universal Lie $\infty$-algebroid over a geometric resolution of $\mathfrak{F}$. For $i=0$, we recover the Nash blowups for a singular foliation, which matches Nash blowups for an affine  variety $W$ when applied to the vector fields tangent to $W$. 
 But for $i\geq 2$, these resolutions have never been introduced before to our knowledge. A consequence of our construction for $i=1$ is that a resolution of any singular foliation can be constructed, which is given by an action of a Lie algebroid whose anchor map is injective on a dense open subset (a result implicit in \cite{MohsenOmar} but not stated as such there). For generic $i$, one obtains a singular foliation which is the image of the anchor map of a Lie $i$-algebroid.

 In general, we must admit that very often the blowup space $\widetilde M_i$'s are not smooth manifolds. However, our singular foliation on $\widetilde M_i$ makes sense and admits leaves that are smooth submanifolds. Also, the blowup spaces $\widetilde M_i$ are analytic varieties if the initial singular foliation admits real analytic generators. Last, we are able to decide when $ \widetilde{M}_i$ is smooth: it suffices to study the properties of some ideal of functions that we describe in the text.
 %In the context of real analytic/holomorphic settings, A. Sertöz Sinan \cite{Ali-Sinan} investigated the conditions under which the $\widetilde M_i$'s are complex manifolds. This corresponds to studying the monoidal transformation of $M$ centered at the singular locus $M_{\mathrm{sing}}$ of $\mathfrak F$, also known as the Hironaka blowup of $M$ along $M_{\mathrm{sing}}$, see Section \ref{sec:smooth} for more details.
 
 %which we understand as : if the closed subset of singular points $M_{\mathrm{sing}}$ is a smooth manifold and the $\ker\dd^{(i)}\subset \Gamma(E_{-i})$'s are finitely generated at the level of sections. Since we are on a geometric resolution $(E,\dd, \rho)$, the $\widetilde M_i$'s are complex manifolds if the closed subset  $M_{\mathrm{sing}}$ of singular points of $\mathfrak F$ is smooth. In this case, $\widetilde M_i$ is the monoidal transformation of $M$ with center $M_{\mathrm{sing}}$ which is also known as the Hironaka blowup of $M$ along $M_{\mathrm{sing}}$. In particular, if $\mathfrak{F}$ admits isolated singularities, the $\widetilde M_i$'s are all smooth.

The paper is structured as follows: In Section \ref{sec:1}, we revisit the concept of singular foliations and their universal Lie $\infty$-algebroids. Section \ref{sec:2} presents the Nash blowup construction for  vector bundle morphisms and discuss smoothness. In Section \ref{Nash-foliations}, we introduce a series of Nash blowups of a singular foliation, indexed by $i\in \mathbb N_0$, followed by the main theorems. In  section \ref{sec:Proofs}, we prove the results of Section \ref{Nash-foliations}. In Section \ref{sec:Examples}, we provide examples of our constructions and demonstrate how the usual notions of blowups for affine varieties can be recovered. Finally, in order to fix notations,  we review in Appendix \ref{appendix} the definition and properties of Grassmann bundles.

\section{Preliminaries: Singular foliations and Lie $\infty$-algebroids}\label{sec:1}
\begin{convention}
 Throughout the article, $\mathcal O$ stands for the sheaf of (smooth, polynomial,
real analytic or holomorphic) functions on (a manifold, affine variety—depending on the context) $M$ and, for a vector
bundle $E\to M$ of constant rank, $\Gamma(E)$ stands for the sheaf of sections of $E$. Also, $\mathbb K\in \{\mathbb R,\mathbb C\}$. The results of this paper apply to the smooth, algebraic, real analytic, and holomorphic contexts, with some adaptations. However, for simplicity, we will primarily work in the smooth setting.
\end{convention}

%\subsection{Singular foliations and universal Lie $\infty$-algebroids}
We refer the reader to \cite{AndroulidakisIakovos,AndroulidakisZambon,Cerveau, Debord,LLS,LLL} for the topic of singular foliations, in particular to \cite{LLS, CLRL} for the notion of universal Lie $\infty$-algebroids. For Lie algebroids, see \cite{Mackenzie-Kirill}.
\subsection{Singular foliations}\label{def}We recall some basic definitions and properties on singular foliations. 
\begin{enumerate}

    \item A \emph{singular foliation} on a manifold $M$
    is a subsheaf $\mathfrak{F}\subseteq\mathfrak{X}(M)$  that fulfills the following conditions,
\begin{enumerate}
     \item \emph{Stability under Lie bracket} : $[\mathfrak{F},\mathfrak{F}]\subseteq \mathfrak{F}$.
     \item $\mathfrak F$ is a \emph{module} over its respective relevant sheaf of functions.
     \item \emph{Locally finitely generateness}\footnote{In the holomorphic case, this condition can be equivalently replaced by the notion of coherent sheaf \cite{zbMATH03423310,Ali-Sinan}.} : every $m\in M$ admits an
open neighborhood $\mathcal{U}$ together with a finite number of vector fields  $X_1, \ldots, X_k \in \mathfrak{X}(\mathcal U)$ such that for
every open subset $\mathcal{V}\subseteq \mathcal U$ the vector fields ${X_1}{|_\mathcal{V}},\ldots, {X_k}{|_\mathcal{V}}$ generates $\mathfrak{F}$ on $\mathcal{V}$ as a module over functions on $\mathcal V$. 
\end{enumerate}

%There are three classes of singular foliations we are particularly interested in.
We are particularly interested in three specific classes of singular foliations, which we  now define.

 \begin{enumerate}
     \item[-] A \emph{locally polynomial/analytic singular foliation} is a singular foliation over a smooth or complex manifold which admits, around each point, generators with polynomial/analytic coefficients  in some local chart.

     \item[-] A \emph{globally finitely generated singular foliation $\mathfrak F\subseteq \mathfrak X(M)$} is  a singular foliation which is generated as an $\mathcal{O}$-submodule of $\mathfrak X(M)$ by finitely many vector fields on $M$.
\item[-] 
    A singular foliation $\mathfrak F$ is \emph{Debord} if it is projective as a module over functions on $M$, equivalently  if and only if there exists a Lie algebroid $(A, \lb_A, \rho)$ such that $\rho(\Gamma(A))=\mathfrak F$ whose anchor is injective on an open dense subset.  In particular, Debord foliations are globally  finitely generated.
    
\end{enumerate}

\item Here are some important features of the above definition in the smooth/real analytic/complex cases, see \cite{LLL}, Section 1.7.
\begin{itemize}
     \item[-] Singular foliation admits leaves : there exists a partition of $M$ into submanifolds called \emph{leaves} such that for all $m\in M$, the image of the evaluation map $\mathfrak F \to T_m M  $  is the tangent space of the leaf through $m$.
     \item[-] \emph{Singular foliations are self-preserving}: the flow $\phi^X_t$ of vector fields $X\in\mathfrak F $, whenever defined, preserves $\mathfrak F$ \cite{Hermann, AndroulidakisIakovos,GarmendiaAlfonso}, i.e., \begin{equation*}
         \forall\, m\in M, \exists\, \epsilon >0\; \text{such that}\; \forall t\in ]-\epsilon,\epsilon[, \; {(\phi^X_t)}_*(\mathfrak F)=\mathfrak F.
     \end{equation*}
\end{itemize}
\end{enumerate}
\subsubsection{Nagano-Sussman theorem}\label{Nagano}
We introduce the following definitions, which are particular cases of a more general notion applied to singular spaces studied in  \cite{Sniatycki}.  Let $S$ be a closed subset of a manifold $N$.
\begin{enumerate}
 \item \label{def:tangent}A \emph{vector field on} $S\subseteq N$  is the restriction to $S$ of a vector fields $Z\in \mathfrak X(N)$  whose flow preserves $S$, i.e., $\phi^{Z}_t(S)\subseteq S$ whenever it makes sense. In that case, we shall say that such a \emph{$Z$ is tangent to $S$}. The set of vector fields on $S$ form a Lie algebra that we  denote as in the usual case by $\mathfrak X(S)$.
        \item The \emph{tangent space} $T_sS$  of $S$ at $s\in  S$ is the evaluation at $s$ of the vector fields on $S$.
        \item \label{ref:LA} We also make sense of the notion of Lie algebroid on a closed subset  $S\subseteq M$ as follows: a \emph{Lie algebroid over $S$} is a locally finitely generated projective Lie-Rinehart algebra over $\mathcal{O}/\mathcal{I}_S$. Here, $\mathcal{I}_S$ is the ideal of vanishing functions on $S$.
\end{enumerate}

\begin{remark}
    Notice that when $S$ is a  submanifold or an analytic subvariety of a complex or real analytic manifold,  this notion of vector field or Lie algebroid on $S$ agrees to the usual case.
\end{remark}

Now,  we recall a crucial theorem that allows to define singular foliations correctly on a closed subset $S$ of a manifold $M$.

\begin{definition}\cite{LLL, LLS}\label{def:leaves}\label{def:SF2} Let $S$ be a closed subset of $N$.
\begin{enumerate}
    \item  A \emph{singular foliation on $S$} is an involutive\footnote{Notice that $\mathfrak{F}$ is generated by the restrictions to $S$ of vector fields on $N$, they are  required to be involutive only after restrictions to $S$.} locally finitely generated $\mathcal{O}_S$-submodule $\mathfrak{F}\subseteq\mathfrak X(S)$.
    \item For $s\in S$, the \emph{leaf} of a singular foliation $\mathfrak{F}$ on $S$ through $s$ is the set 
\begin{equation}
        L_s:=\left\{\phi_{t_1}^{Z_1}\circ \phi_{t_2}^{Z_2}\circ \cdots \circ \phi_{t_k}^{Z_k}(s),\; t_1, \ldots, t_k\in \mathbb R\; \right\}
    \end{equation}
 Above, $Z_1,\ldots,Z_k$ are vector fields on $N$  whose restrictions are in $ \mathfrak{F}$. We implicitly assume that the flows are defined.
\end{enumerate}
\end{definition}
\begin{remark}
    From  Definition \ref{def:leaves}, it is easily checked that being in the same leaf is an equivalence relation on $S$, hence the leaves induce a partition of $S$.
\end{remark}

The notion of leaves of singular foliation on $S\subseteq N$ is justified by the following theorem that generalizes the Stefan-Sussman theorem \cite{Stepan1,Stepan2}, that says the leaves are smooth manifolds.

\begin{theorem}\label{thm:Sussmann_SP}Let $\mathfrak{F}$ be a singular foliation on a closed subset  $S\subseteq N$. The leaves $\mathfrak F$ form a partition of $S$ into connected manifolds,  immersed as submanifolds of $N$.
\end{theorem}
The explanation of this result is based on a very strong theorem known as the Nagano-Sussman theorem \cite{Nagano}. This theorem, widely used in control theory, provides a very strong result regarding the smoothness of the orbits of a finite number of vector fields on a manifold  without any assumptions. 
       \begin{theorem}[Nagano–Sussmann] Let $\mathcal{V}\subseteq \mathfrak X(N)$ be a locally finitely generated $\mathcal{O}_N$-submodule of vector fields on a manifold $N$. For every $\ell\in N$, the set
       $$\left\{\phi_{t_1}^{Z_1}\circ \phi_{t_2}^{Z_2}\circ \cdots \circ \phi_{t_n}^{Z_n}(\ell),\; t_1, \ldots, t_n\in \mathbb R,\; Z_1,\ldots, Z_n\in \mathcal{V},\;n\in \mathbb N\right\}$$

     is a connected immersed submanifold of $N$.
       \end{theorem}

\begin{proof}[Proof (of Theorem \ref{thm:Sussmann_SP})]

   For simplicity, let us assume that $\mathfrak{F}\subseteq \mathfrak X(S)$ is globally finitely generated (the general case is left to the reader). Let $\xi_1,\ldots, \xi_k$ be generators for $\mathfrak F$. By definition, the $\xi_i$'s are the restrictions to $S$ of vector fields $Z_i$'s on $N$ whose flows $\phi^{Z_i}_t$ preserves $S$, i.e.,   $\phi^{Z_i}_t(S)\subseteq S$ where the flows are defined. By Nagano-Sussman theorem, the orbits generated by the vector fields  $Z_1,\ldots, Z_k\in \mathfrak X(N)$  are immersed submanifolds of $N$. By assumption, the orbits through a point of $S$ of the $\xi_i$'s coincide with the orbits of the $Z_i$'s % the orbit that contains a point of $S$ 
   and are included in $S$. This completes the proof.
\end{proof}

 \subsection{Universal Lie $\infty$-algebroid of a singular foliation}\label{sec:universal}Let us recall the notion of universal  Lie $\infty$-algebroid of a singular foliation. Let $\mathfrak{F}\subseteq\mathfrak{X}(M)$ be a submodule.
\begin{enumerate}
\item \label{def: geom.resol} A complex of vector bundles $(E, \dd, \rho)$
    $$   \xymatrix{  \ar[r] & E_{-i-1} \ar[r]^{{\dd^{(i+1)}}} \ar[d] & 
     E_{ -i} \ar[r]^{{\dd^{(i)}}}
     \ar[d] & E_{-i+1} \ar[r] \ar[d] & \ar@{..}[r] & \ar[r]^{{\dd^{(2)}}}& E_{-1} \ar[r]^{\rho} \ar[d]& TM \ar[d] \\ 
      \ar@{=}[r] & \ar@{=}[r] M  &  \ar@{=}[r] M 
      &  \ar@{=}[r] M  &\ar@{..}[r] & \ar@{=}[r]   &  \ar@{=}[r] M  & M}
    $$
    is said 
 to be a \emph{geometric resolution of  $\mathfrak F$} if the following complex is an exact sequence of sheaves:
      \begin{equation}
           \label{eq:exact}{\longrightarrow} \Gamma({E_{ -i-1}})
     \stackrel{\dd^{(i+1)}}{\longrightarrow} \Gamma({E_{-i}})
     \stackrel{\dd^{(i)}}{\longrightarrow}{\Gamma(E_{-i+1})}{\longrightarrow}\cdots  {\longrightarrow} \Gamma(E_{-1})  
     \stackrel{\rho}{\longrightarrow} \mathfrak F.
      \end{equation} A geometric resolution is said to be of \emph{finite length $n\in \mathbb N_0$} if $E_{-i}=0$ for all $i\geq n+1$. Also,  $(E, \dd, \rho)$ is said to be \emph{minimal} at a point $x\in M$ if, for all $i\geq 2$, the linear maps $\dd^{(i)}{|_x}\colon {E_{-i}}_{|_x}\longrightarrow  {E _{-i+1}}_{|_x}$ vanish. 
      
%Recall that it exists in many contexts, e.g., singular foliationwith polynomial generators on $\mathbb R^n$ or $\mathbb C^n$ or locally real analytic singular foliation on a compactmanifold (see \cite{LLS}).
\item\label{almost}  An \emph{graded almost  Lie algebroid over $M$} is the datum of a complex $(E,\dd=\ell_1,\rho)$ of vector bundles over $M$ equipped with a graded symmetric degree $+1$ $\mathbb{K}$-bilinear \emph{bracket} $$ \ell_2 \colon \Gamma(E)\odot\Gamma(E)\rightarrow\Gamma(E) $$ such that:
\begin{enumerate}
\item  $\ell_2$ satisfies the \emph{Leibniz identity} with respect to  $\rho \colon \Gamma( E_{-1}) \longrightarrow \mathfrak{X}(M) $, i.e.,
\begin{equation*}
    \ell_2(x, fy) = f\ell_2(x,y)+\rho(x)[f]y
\end{equation*}
for all $x\in \Gamma(E_{-1}),y\in \Gamma(E)$ and $f\in\mathcal{O}$.
\item $\ell_1$ is degree $+1 $-derivation of $\ell_2 $, i.e., for all  $x \in \Gamma( E_{-i}), y \in \Gamma(E)$:
$$ \ell_1 (\ell_2(x,y)) +  \ell_2( \ell_1( x),y) + (-1)^{i}  \ell_2( x  , \ell_1( y)) = 0,$$ 

\item $\rho$ is a morphism, i.e., for all  $x,y \in \Gamma(E_{-1})$
$$ \rho (\ell_2 (x,y)) = [\rho (x), \rho (y)].$$
 \end{enumerate}
The $\mathcal{O}$-linear map $\rho$ is called the \emph{anchor map}, and $\ell_1$ the \emph{differential}.
\item  A \emph{Lie $\infty$-algebroid over $M$} is the datum of a sequence $E = (E_{-i} ),\,1\leq i<\infty$ of vector bundles over $M$ together with a structure of Lie $\infty$-algebra $(\ell_k)_{k\geq 1}$ on the sheaf of sections of $E$ and a vector bundle morphism, $\rho\colon E_{-1}\rightarrow TM$, called \emph{anchor map} such that the $k$-ary brackets $\ell_k,\,k\neq 2$ are $\mathcal O$-multilinear and such that
\begin{equation}
  \ell_2(e_1, f e_2) =\rho(e_1)[f]e_2 + f\ell_2(e_1, e_2) 
\end{equation}
for all $e_1\in \Gamma(E_{-1} ), e_2\in\Gamma(E_\bullet)$ and $f\in\mathcal O$. The sequence
\begin{equation}\label{linpart}
        \xymatrix{\cdots\ar[r]^{\ell_1}&E_{-2}\ar[r]^{\ell_1}&E_{-1}\ar[r]^{{\rho}}&{TM},}
    \end{equation}
is a complex called the \emph{linear part} of the Lie $\infty$-algebroid. \\

Notice that given a Lie $\infty$-algebroid $(E, (\ell_k)_{k\geq 1},\rho)$, the quadruple $(E,\dd=\ell_1,\ell_2, \rho)$  is a graded almost  Lie algebroid over $M$.

The following theorem is important, see Section 2 in \cite{LLS} or \cite{CLRL} for more details. \begin{theorem}
Let $\mathfrak{F}$ be a singular foliation over $M$.
Any geometric resolution  of ~$\mathfrak F$ 
\begin{equation}
    \label{eq:resolutions}
\cdots \stackrel{\dd} \longrightarrow E_{-3} \stackrel{\dd}{\longrightarrow} E_{-2} \stackrel{\dd}{\longrightarrow} E_{-1} \stackrel{\rho}{\longrightarrow} TM \end{equation}
  comes equipped with a Lie $\infty $-algebroid structure whose unary bracket is $\dd $ and whose anchor map is $\rho$ (in particular $\rho(\Gamma(E_{-1}))=\mathfrak{F}$). Such a Lie $\infty $-algebroid structure is unique up to homotopy and is called a \emph{universal Lie $\infty $-algebroid of} $\mathfrak F$.
  \end{theorem}
  In particular, this Lie $\infty$-algebroid structure can be truncated to a graded almost  Lie algebroid for $\mathfrak{F}$.
  \item \label{item:infty-isotropy}  Let $\left(E_\bullet,\ell_\bullet, \rho\right)$  a universal Lie $\infty$-algebroid of a singular foliation $\mathfrak{F}$. For every point $x\in M$, 
  
  \begin{enumerate}
  \item We let $H^\bullet(\mathfrak F,x)=\oplus_{i\geq 1}H^{-i}(\mathfrak{F},x)$ be the cohomology of the complex \eqref{eq:resolutions}. The cohomology groups $H^\bullet(\mathfrak F,x)$ do not depend on the choice of a geometric resolution of $\mathfrak{F}$. Notice that when the complex \eqref{eq:resolutions} is minimal at $x$, $H^{-i}(\mathfrak{F},x)\simeq E_{-i}|_x$ for every $i\geq 1$.
      \item The $1$-ary and the $2$-ary brackets restrict to the graded vector space $$\left(\bigoplus_{i\geq 2} {E_{-i}}_{|_x}\right)\oplus \ker (\rho_{x})$$
and equip the latter with a graded almost  Lie $\infty$-algebra structure as follows : for every $k\in \{1,2\}$,
\begin{equation*}
    \{x_1,\ldots,x_k\}_k :=\ell_k(s_1,\ldots,s_k)_{|_x}
\end{equation*}
for all $x_1,\ldots,x_k\in ev(E,x) $ and $s_1,\ldots,s_k\in \Gamma(E)$ sections of $E$ such that ${s_i}(x)=x_i$ with $i=1,\ldots,k$.
  \end{enumerate}
The bracket $\{\cdot\,,\cdot\,\}_2$ induces a graded Lie algebra on $H^\bullet(\mathfrak F,x)$. In particular, the $2$-ary bracket $\{\cdot\,,\cdot\}_2$ satisfies the Jacobi identity on $H^{-1}(\mathfrak F,x)=\frac{\ker(\rho_x)}{\mathrm{im}(\dd^{(2)}_{x})}$, and equips the latter with a Lie algebra structure.

  \item  Let $(M,\mathfrak{F})$ be a singular foliation, let $\mathcal I_x := \left\lbrace f \in
C^\infty (M )\mid  f (x) = 0\right\rbrace$ and $\mathfrak F(x):= \left\lbrace X \in\mathfrak F \mid  X(x) =0\right\rbrace$. The quotient   $\mathfrak{g}_x=\frac{\mathfrak{F}(x)}{\mathcal{I}_x\mathfrak{F}}$ is a Lie algebra and is called the isotropy Lie algebra of $\mathfrak{F}$ at $x$. A point $x\in M$ is said to be a \emph{regular point of $\mathfrak F$} if $\mathfrak g_x=\{0\}$, otherwise we say that $x$ is a \emph{singular point}. The set of regular points of $\mathfrak F$ is denoted by $M_{\mathrm{reg},\mathfrak F}$.\begin{lemma}\cite{LLS}\label{lemma:M_reg}
Let $(E,\ell_\bullet, \rho)$ be a universal  Lie $\infty$-algebroid of $\mathfrak{F}$. Consider its underlying geometric resolution 
\begin{equation*}
   % \label{eq:resolutions2}
(E, \dd, \rho) :\quad\cdots \stackrel{\ell_1=\dd^{(4)}} \longrightarrow E_{-3} \stackrel{\ell_1=\dd^{(3)}}{\longrightarrow} E_{-2} \stackrel{\ell_1=\dd^{(2)}}{\longrightarrow} E_{-1} \stackrel{\rho=\dd^{(1)}}{\longrightarrow} TM.\end{equation*} Then, 
\begin{enumerate}
 \item for all $x\in M$, we have $   H^{-1}(\mathfrak F,x)\simeq \mathfrak{g}_x$ as Lie algebras;
    \item the subset of regular points of $\mathfrak F$ in $M$ satisfies \begin{align*}
    M_{\mathrm{reg}, \mathfrak{F}}&=
    \{x\in M \mid \mathrm{rk}(\dd^{(2)}_x)=\dim(\ker\rho_x)\}\\&=
    \{x\in M \mid H^{-i}(\mathfrak{F}, x)=0, \forall i\geq 1\},
    \end{align*}
    $M_{\mathrm{reg}, \mathfrak  F}$ is open and dense in $M$;

\item the restriction of the foliation $\mathfrak{F}$ to $M_{\mathrm{reg}, \mathfrak  F}$ is
the set of sections of a subbundle of TM, i.e., is a regular foliation;
\item \label{iterated-ranks} For every $i\geq 0$, the dimension of $\mathrm{im}\left(\dd^{(i+1)}\right)$ is locally constant on $M_{\mathrm{reg}, \mathfrak  F}$. Moreover, if $r$ the dimension of a regular leaf, then $\mathrm{im}(\dd^{(i+1)})$ is of codimension  $$r_i={\sum_{j=1}^{i-1}}(-1)^{j+1}\mathrm{rk}(E_{-j})+(-1)^{i+1}r, \;\text{for}\; i\geq 1$$ in $E_{-i}$ or $r_0=\dim M-r$, with  $E_0:=TM$;
\item if $(E,\dd,\rho)$ is of finite length, then all the regular leaves have the same dimension.
\end{enumerate}
\end{lemma}
\end{enumerate}
In the sequel, we assume that a geometric resolution of finite length exists. Under these assumptions, all the regular leaves have the same dimension. We denote by $r$ the common dimension of the regular leaves.
\section{Blowup procedures}\label{sec:2}

\subsection{Blowup of vector bundle morphisms.}

Firstly, let us explain a general construction on morphisms of vector bundles that we refer as Nash blowup. For an open subset $\mathcal V \subseteq M$, we shall denote by $\mathcal{V}^\mathbb N$ the $\mathcal V$-valued sequences of points $(x_n)$ indexed by $\mathbb{N}$. We direct the reader to Appendix \ref{appendix} for conventions and notations regarding Grassmannians.

\subsubsection{The Nash blowup of a vector bundle morphism}

Let $E, F$ be vector bundles over $M$ and $$\xymatrix{F\ar[rr]^\dd\ar[rd] &&E\ar[ld]\\&M&&}$$
a morphism of vector bundles over the identity. In the smooth case, we  assume that $\dd$ is of constant rank on an open dense subset $M_{\mathrm{reg},\dd}\subset M$, i.e., the dimensions of  $\mathrm{im}(\dd_x)$ or $\ker (\dd_x)$ are constant for $x\in M_{\mathrm{reg},\dd}$, called the \emph{regular} part. Let $q$ be the co-dimension of $\mathrm{im}(\dd_x)\subseteq E_x$ for a point  $x\in M_{\mathrm{reg},\dd}$. Notice that for every $x\in M_{\mathrm{reg},\dd}$, $\mathrm{im} (\dd_x)$ is a point of the Grassmannian  $\mathrm{Gr}_{-q}(E_x)$  of  vector subspaces of $E_x$ of co-dimension $q$. Also,  $\ker(\dd_x)$ is a point of the Grassmanian $\mathrm{Gr}_{-(\mathrm{rk}(F)-q)}(F_x)$ of  vector subspaces of $F_x$ of co-dimension $\mathrm{rk}(F)-q$. We consider the natural section of the Grassmann bundle $\Pi\colon \mathrm{Gr}_{-q}(E)\longrightarrow M$ which is defined on $M_{\mathrm{reg},\dd}$ by: \begin{equation}\label{eq:Nash-vb-map}
    \sigma\colon M_{\mathrm{reg},\dd}\longrightarrow \mathrm{Gr}_{-q}(E),\, x\longmapsto \mathrm{im}(\dd_x).
\end{equation}
Then we define the \emph{Nash blowup space of $M$ along $\dd$} to be the closure $\widetilde{M}:= \overline{\sigma(M_{\mathrm{reg},\dd})}$ of the image of the section $\sigma$ in $\mathrm{Gr}_{-q}(E)$. It comes together with the projection $\pi\colon \widetilde{M}\longrightarrow M$, where $\pi$ denotes  the restriction of $\Pi\colon \mathrm{Gr}_{-q}(E)\longrightarrow M$ to  $\widetilde{M}$.
\begin{remark}
A detailed line-by-line comparison provides the construction of Nash as presented in \cite{Ali-Sinan, JPTS}, for a coherent sheaf of modules $\mathcal{K}$ over $\mathcal{O}$, i.e.,  a sheaf of $\mathcal{O}$-modules $\mathcal{K}$ such that for every
$m\in M$ there is an open neighborhood $\mathcal{U}$ of $m$ and an exact
sequence $$\mathcal{O}^n_\mathcal{U}\to \mathcal{O}^k_\mathcal{U}\to\mathcal{ K}_\mathcal{U}\to 0$$
for some integers $n$ and $k$. In Section \ref{Nash-foliations}, we apply the construction \eqref{eq:Nash-vb-map} to the case where $\ker \dd\subseteq \Gamma(F)$ is locally finitely generated, that is to say when the sheafffication of the image $\mathrm{im}(\dd)\subseteq \Gamma(E)$ is a coherent sheaf.
\end{remark}

\begin{remark}
 Intuitively, for $x\in M$, $\pi^{-1}(x)= \widetilde M\cap \Pi^{-1}(x)$ is the set of  all possible limits $\mathrm{Gr}_{-q}(E)$ of the images $\mathrm{im}(\dd_y)$ when $y\in M_{\mathrm{reg},\dd}$ converges to $x$.
\end{remark}
\begin{remark}
One can make  a similar construction with  the kernel of $\dd$.
\end{remark}
Here is an immediate property of that construction.
\begin{proposition}\label{prop:bi}Let $F\stackrel{\dd}{\longrightarrow}E$ be a vector bundle morphism over $M$. The projection $\pi\colon \widetilde M\rightarrow M$ has the following property:
 \begin{enumerate}
\item 
$\pi$ is proper and surjective. In particular, for each point $x\in M$, the fiber $\pi^{-1}(x)$ is non-empty.
\item {For every $x\in M$ and $V\in \pi^{-1}(x)$, one has $\mathrm{im}(\dd_x)\subseteq V$.}
\item  For every $x\in M_{\mathrm{reg},\dd}$,   $\pi^{-1}(x)=\mathrm{im}(\dd_x)$ is reduced to a point in $\mathrm{Gr}_{-q}(E)$. Also, $\pi^{-1}(M_{\mathrm{reg},\dd})$
is a manifold \footnote{Manifold is to be understood as quasi-projective when $M$ is quasi-projective.} and the restriction $\pi\colon \pi^{-1}(M_{\mathrm{reg},\dd})\longrightarrow M_{\mathrm{reg},\dd}$ is invertible\footnote{Invertible here means: diffeomorphism, in the smooth case, bi-holomorphism, in the complex case%, biregular in the quasi-projective case
.} in the smooth and holomorphic contexts.
\end{enumerate}
\end{proposition}
\begin{proof} Properness derives from the fact that the projection $\Pi$ admits compact fibers. For any $x\in M$, choose $\mathcal U\subset M$ an open neighborhood of $x$ that trivializes $E\longrightarrow M$ over $\mathcal{U}$. Then, $\mathrm{Gr}_{-q}(E)\simeq \mathcal{U}\times\mathrm{Gr}_{-r}(\mathbb{K}^{\mathrm{rk}(E)})$. Notice that,     \begin{equation*}
        \pi^{-1}(x)=\left\{V\subset E_{_x}\;\middle|\; \exists\, (x_n)\in M_{\mathrm{reg},\dd}^{\mathbb{N}},\, \text{such that},\,\;\mathrm{im}(\dd_{x_n})\underset{n \to +\infty}{\longrightarrow} V\; as \; x_n\underset{n \to +\infty}{\longrightarrow}x\right\}.
    \end{equation*}For any sequence $(x_n)$ in $\left(M_{\mathrm{reg},\dd}\cap\mathcal{U}\right)^\mathbb{N}$ that converges to $x$, we can extract a sequence $(x_{\varphi(n)})$ such that $n\mapsto \mathrm{im}(\dd_{x_{\varphi(n)}})\in \mathrm{Gr}_{-r}(\mathbb{K}^{\mathrm{rk}(E)})$ has a limit $V$, since the Grassmannian manifold $\mathrm{Gr}_{-r}(\mathbb{K}^{\mathrm{rk}(E)})$ is compact. Hence, $\pi^{-1}(x)\neq  	\emptyset$ and $\pi$ is onto. This proves item 1.

Let us show item 2. {Let $V\in \pi^{-1}(x)$ and $(x_n)\in (M_{\mathrm{reg},\dd})^{\mathbb{N}}$ such that $x_n\underset{n \to +\infty}{\longrightarrow}x$\;  and  \; $\mathrm{im}(\dd_{x_n})\underset{n \to +\infty}{\longrightarrow} V$. Let $v\in\mathrm{im}(\dd_{x})$. We have $v=\dd_{x}u$ for some $u\in {F_x}$. Choose a (local) section $\widetilde{u}$ of $F$ through $u$. By continuity,  $\dd_{x_n}\widetilde{u}(x_n)\underset{n \to +\infty}{\longrightarrow}\dd_{x}u$, hence $\dd_{x}u\in V$. Thus, $\mathrm{im}(\dd_{x})\subseteq V$.}

In particular, if $x\in M_{\mathrm{reg}, \dd}$ and $V\in \pi^{-1}(x)$ one has $\mathrm{im}(\dd_{x})= V$ since $\dim V= \dim(\mathrm{im}(\dd_{x}))$. Therefore, $\pi^{-1}(M_{\mathrm{reg},\dd})$  is the image of the map $\sigma$ on $M_{\mathrm{reg},\dd}$, it is isomorphic/biholomorphic to $M_{\mathrm{reg},\dd}$. This proves item 3.
\end{proof}

\begin{remark}
\label{rmk:invariance}
Let $A,B,C,E,F$ be vector bundles over  $M$.
The Nash blow up space of a vector bundle morphism $ \dd : F \to E$ coincides with
the Nash blowup of the vector bundle morphism 
 $$ \begin{array}{rcl}  A \oplus B \oplus F  &\to & B \oplus E \oplus C\\  (a,b,f)  & \mapsto & (b,\dd (f),0) \end{array} .$$
The result is left to the reader.
\end{remark}

In Section \ref{Nash-foliations}, we apply the constructions above to a sequence of vector bundle morphisms which are all of constant rank on an open dense subset.

\subsubsection{On the smoothness of the Nash blowup and monoidal transformations}\label{sec:smooth}

We warn the reader not to confuse two uses of the word “smooth”. An analytic subset of $ \mathbb C^N$ or $ \mathbb R^N$ is said to be smooth when it admits no singular point. The word “smooth” is also used to say that we work within the context of smooth differential geometry, using smooth manifolds on which functions are of class $ \mathcal C^\infty$. Notice that a smooth analytic variety is also a smooth manifold. The context should, however, prevent all confusions.

\vspace{0.5cm}

The “blow-up” Nash construction in the previous section is of a type that algebraic geometers. call monoidal transformations, also known as Hironaka blowups \cite{Hironaka-Rossi}.
Several authors \cite{Nobile, Ali-Sinan} or \cite{gonzalez2009nash} have used this point of view to  study the smoothness of the blowup of a singular foliation, and to compute explicitly the blowup space $\widetilde M$.
%This can be reduced to the smoothness of some monoidal tr.  

%It is important to note, however, that the set of regular points he considered is larger than ours, as he included points around which the submodule $\mathrm{im}(\dd) \subset \Gamma(E)$ is projective.

Let us recall what monoidal transformations are.
%construction.
Denote by  $ \mathcal O$ the sheaf of holomorphic or real analytic functions on a holomorphic or real analytic manifold $M$. Let $ \mathcal I \subset \mathcal O$ be a locally finitely generated sub-sheaf of  $\mathcal O $. Denote by $ Z(\mathcal I) \subset M$  its zero locus (= the subset of all points where all functions in $ \mathcal I$ vanish). The subset $ Z(\mathcal I)$ is a closed subset for the usual topology, and  $M \backslash  Z(\mathcal I)$ is an open dense subset of $M$.
%in the holomorphic, real analytic or algebraic settings. In the smooth case, we will also assume this is to be true. 

We call \emph{monoidal transformation of $M$ with respect to $ \mathcal I$} (in the sense of \cite{Hironaka-Rossi}) 
%of an open subset $\mathcal{U}\subseteq M$ with 
%center $Z\subsetneq \mathcal{U}$ a closed subset whose points are the zero locus of the ideal $\mathcal I$ generated by real analytic/ holomorphic functions $\varphi_1,\ldots, \varphi_n\in \mathcal{O}_{\mathcal{U}}$,  is 
the pair $ (\mathrm{Bl}_\mathcal I (M),\pi)$ constructed as follows.
Let $ \mathcal U \subset M$ be an open subset such that $ \mathcal I_\mathcal U$ is generated by a finite family $ \varphi_1, \dots, \varphi_n$.
\begin{enumerate}
    \item Define a map $$H\colon \mathcal{U}\setminus Z(\mathcal I) \longrightarrow \mathbb P^{n-1}, \;x\to [\varphi_1(x)\colon\cdots \colon \varphi_n(x)].$$

    \item Then we  consider the ideal $\mathcal{J}$ of (projective) functions on $ \mathcal U\times \mathbb P^{n-1}$ generated by  $$\left(\left(x, [\xi_1\colon\cdots\colon \xi_n] \right)\mapsto \xi_i\varphi_j(x)-\xi_j\varphi_i(x)\right)_{i\neq j=1,\ldots,n}$$  where $[\xi_1\colon\cdots\colon \xi_n]$ are the homogeneous coordinates on $\mathbb P^{n-1}$.
\end{enumerate}
The closure $\mathrm{Bl}_\mathcal I(\mathcal{U})\subset  \mathcal U\times \mathbb P^{n-1}$ of the graph of $H$ in  $ \mathcal U\times \mathbb P^{n-1}$ is an irreducible component of the zero locus of the ideal $\mathcal{J}$, and is in particular an irreducible analytic subvariety.  
The natural projection $ \pi_\mathcal U: \mathrm{Bl}_\mathcal I(\mathcal{U}) \to \mathcal U$ is  a real analytic, or holomorphic proper map and restricts to an invertible map in the relevant category $ \pi^{-1}(  \mathcal{U}\setminus Z(\mathcal I)) \to \mathcal{U}\setminus Z(\mathcal I)$. 
Last, one can check that the pair $\left(\mathrm{Bl}_{\mathcal I}(\mathcal{U}),\pi_\mathcal U\right)$ does not depend on the choice of local generators $ \varphi_1, \dots, \varphi_n$, see e.g., Lemma 9.16 in \cite{Gathmann}. This explains the notation. 
%As a consequence, these locally defined sets can be glued to define a pair  $\left(\mathrm{Bl}_{\mathcal I}(M),\pi\right)$, with $ \pi: \mathrm{Bl}_{\mathcal I}(M) \to M$ a surjective proper map. Also,  $\pi^{-1}(M \setminus Z(\mathcal I)) $ admits a  natural manifold structure, and the restriction of $\pi $ to that manifold is invertible in the relevant category. 

\vspace{0.5cm}
%Let us for a moment put the smooth setting apart. 
Let us explain how the monoidal transformation depends on the ideal. Assume that we are given two ideals  $ \mathcal I$ and $ \mathcal J$ over $M$. It is natural to ask what are the conditions that guarantee the existence or the dotted arrow below:
 \begin{equation}\label{dia:dottedarrow}  \xymatrix{  \mathrm{Bl}_{\mathcal I}(M) \ar@{.>}[r]\ar[d]_\pi&\ar[d]^\pi \mathrm{Bl}_{\mathcal J}(M)\\ M \ar[r]^{=}& M } .  \end{equation}
 If it exists, then it has to be  unique. 
The theorem proved by Moody in \cite{Moody} gives a definitive answer,  by stating that the following two conditions are equivalent:
 \begin{enumerate}
 \item[(i)] the dotted arrow in \eqref{dia:dottedarrow} exists, and
 \item[(ii)] There exists an integer $n$ and a finitely generated sub-$\mathcal O$-module $\mathcal K $ in the sheaf of the fraction field $\hat{\mathcal O }$ of $ \mathcal O$ such that $   \mathcal K \cdot \mathcal J =\mathcal I^n $.
 \end{enumerate}
We will call this equivalence the \emph{Moody criteria}.
 
\begin{remark}
\label{rmk:glueing}
In particular, given an open cover $(\mathcal U_i)_{i \in I} $ of $M$ by coordinates open subsets and a family $ \mathcal I_i \subset \mathcal O(\mathcal U_i)$ of finitely generated ideals such that, for every $ i,j \in I$, there exists an element $\phi_{ij} \in \hat{\mathcal O}(\mathcal U_i\cap \mathcal U_j) $ such that $ \mathcal I_i|_{\mathcal U_i\cap\mathcal U_j}  = \phi_{ij} \mathcal I_j|_{\mathcal U_i\cap \mathcal U_j} $, then the locally defined monoidal transformations $({\mathrm{Bl}}_{\mathcal I_i}\mathcal U_i , \pi) $ coincide in $ \mathcal U_i \cap \mathcal U_j$ and define, globally, a pair $( \tilde{M},\pi) $. Moreover, any other family $ (\mathcal U_j',\mathcal I_j')$ such that $ \mathcal I_i|_{\mathcal U_i\cap \mathcal U_j'}  = \phi_{ij} \mathcal I_j'|_{\mathcal U_i\cap\mathcal U_j'} $ for some $ \phi_{ij} \in \hat {\mathcal O}(\mathcal U_i \cap \mathcal U_j')$ will define the same pair $ (\tilde{M},\pi)$.
\end{remark}

\vspace{0.5cm}
Let us apply these general facts to the situation of a vector bundle morphism $\dd\colon F\to E$ which can be either holomorphic or real analytic.

There is an open subset $ M_{\mathrm{reg},\dd}$ on which $\dd$ has constant rank. We denote by $ k$ this rank.  
The integer $k$ also admits the following characterisation. 
Any point $m$ admits a connected neighborhood $\mathcal U $ on which the $ \hat{\mathcal O}(\mathcal U)$-module $ \hat{\mathcal O}(\mathcal U) {\mathrm{im}}(\dd) $ generated by the image of $ \dd$ is a vector space whose dimension over $ \hat{\mathcal O}(\mathcal U)$ is $ k$.  
This characterisation has the following consequence: we say that a family $ e_\bullet = e_1, \dots, e_k$ of local sections of $E$ over such an open subset $ \mathcal U$ is \emph{admissible} if $\dd(e_1), \dots, \dd(e_k)$ are independent on an open subset of $\mathcal U $.  Equivalently, it means that  $\dd(e_1), \dots, \dd(e_k)$ is a basis of $ \hat{\mathcal O}(\mathcal U) {\mathrm{im}}(\dd) $. For any two admissible families
$ e_\bullet = e_1, \dots, e_k$  and $ e_\bullet' = e_1', \dots, e_k'$ therefore, there exists a matrix valued in $ \hat{\mathcal O}(\mathcal U)$ such that:
 \begin{equation} \label{eq:matrix}  \begin{pmatrix}
     e_ 1 \\ \vdots \\ e_k
 \end{pmatrix}  \, = \begin{pmatrix}
     m_{1,1} &\ldots &  m_{k,1}\\ \vdots &&\vdots\\ m_{1,k} &\ldots & m_{k,k}
 \end{pmatrix}  \, \begin{pmatrix}
     e_1' \\ \vdots \\ e_k'
 \end{pmatrix}  .\end{equation}
Now, given a coordinate open set $ \mathcal U \subset M$ and an admissible family $e_{\bullet} := e_1, \dots, e_k \in \Gamma(F)$, one can consider the ideal $ \mathcal I_{{\mathrm{im}}(\dd(e_\bullet))}$ generated by  $ \langle \alpha , \dd(e_1)\wedge \cdots \wedge \dd(e_k)\rangle $ for $ \alpha \in \Gamma(\wedge^k E^*) $. 
  Equivalently, $I_{{\mathrm{im}}(\dd(e_\bullet))}$ can also be seen as the ideal generated by all $ k \times k$ minors of the $ k \times {\mathrm{rk}}(F)$-matrix representing the vectors $\dd(e_1), \dots ,\dd(e_k)  $ on a given   trivialization of $F$ on $ \mathcal U$.
%\end{enumerate}
For any two admissible families $ e_\bullet$ and $ e_\bullet '$ defined on two such  coordinate open sets $ \mathcal U$  and $ \mathcal U'$, respectively, we have on $ \mathcal U \cap \mathcal U'$
 \begin{equation}\label{eq:changeadmissible}  I_{{\mathrm{im}}(\dd(e_\bullet))} = \phi_{e_\bullet / e_\bullet'}   I_{{\mathrm{im}}(\dd(e_\bullet'))}   \end{equation}
 where $\phi_{e_\bullet / e_\bullet'}  \in \hat{\mathcal O}(\mathcal U \cap \mathcal U')$ is the determinant of the $k \times k$ matrix as in Equation \eqref{eq:matrix}.
One can then cover  $M$ by open coordinate neighborhoods $(\mathcal U_i)_{i \in I} $, then choose an admissible family on each one of them. 
The construction in Remark \ref{rmk:glueing} applies and yields a pair $ (\tilde{M},\pi)$.

\begin{lemma}
The pair $ (\tilde{M},\pi)$ satisfies the following properties.
\begin{enumerate}
\item In a neighborhood $ \mathcal U$ of every point of $M$, it is given by the monoidal transformation with respect to the ideal generated by the $k \times k$ minors of $\dd(e_1), \dots, \dd(e_k)$ with $ e_1, \dots, e_k$ being \emph{any} admissible family. 
\item In particular, each point has a neighborhood near which it is an analytic variety, and $ \pi$ is a holomorphic or real analytic proper map.
\item Last, $ \pi^{-1}(M_{\mathrm{reg},\dd})$ admits a natural manifold structure to which the restriction of $ \pi$ is invertible.
\end{enumerate}
\end{lemma}
\begin{proof}
The first and second items hold by construction. The last item follows from the fact that for any $m \in M_{\mathrm{reg},\dd}$, there exists at least one admissible family $e_\bullet=(e_1, \dots,e_k)$ such that the vectors $\dd(e_1), \dots, \dd(e_k)$ are independent at the point $m$. In particular, $m $ does not belong to the zero locus of the ideal $I_{{\mathrm{im}}(\dd(e_\bullet))} $.
\end{proof}

Here is the main result of this discussion, whose proof is delayed to the end of the section.
It shares some similarity with Theorem 1 in \cite{zbMATH05707343}, that deals with Nash blow-up of affine varieties.

 \begin{proposition}
 \label{prop:thisisNash}
The pair  $ (\tilde{M},\pi)$ coincides with the Nash blowup. 
 \end{proposition}

A practical consequence of Proposition \ref{prop:thisisNash} consists in allowing to apply the existing literature about monoidal transformations with respect to ideals, e.g., criterions for smoothness  for monomial ideals \cite{zbMATH01623323}, or for tame monomial ideals \cite{FABER20111805}. Also, notice that Proposition \ref{prop:thisisNash}  still applies in the context of smooth differential geometry, provided that each point of $M$ admits local coordinates on which $ \dd$ is given by a matrix whose coefficients are real analytic. These coordinates even do not need to glue in a real-analytic manner. Here is an application about the smoothness of $ \tilde{M}$.

\begin{corollary}
Let $M$ be a complex manifold, and $ \dd : E \to F$ a vector bundle morphism.
If $ M_{\mathrm{sing}}$ is a smooth submanifold of $M $, and if every point of $m$ admits local  admissible sections $ e_\bullet = e_1, \dots, e_k$ such that there exists $ \chi \in \hat{\mathcal O}$ and  $n \in \mathbb N$ that satisfy
  $$  I_{{\mathrm{im}}(\dd(e_\bullet))}  = \chi \mathcal I_{\mathrm{sing}}^n , $$
  where $ \mathcal I_{\mathrm{sing}}$ is the ideal of functions vanishing on the singular locus,  
then $ \tilde{M} $ is a smooth manifold.
\end{corollary}
\begin{proof}
This follows from criteria of Moody recalled above and the fact that the monoidal transformation with respect to an ideal and its powers are the same, together with the fact that the monoidal transformation with respect to the ideal of functions vanishing on a smooth submanifold is a smooth manifold.
\end{proof}

%XXXXXXXXXXXXXXXXXXXXXXXXXXXXXXXXXXXXXXXXXXXXXXXXXXXXXXXXXX

\begin{proof}[Proof of Proposition \ref{prop:thisisNash}]
We only give a sketch of the proof, since it is equivalent to the one given in 
\cite{Nobile,Ali-Sinan}. Let $\mathcal{U}$ be an open subset of $M$ that trivializes both $F$ and $E$. Let $(u_1, \ldots, u_d)$ and $(e_1, \ldots, e_{d'})$ be local frames of $F$ and $E$, respectively, on $\mathcal{U}$. The sections $\dd(u_1),\ldots,\dd(u_d)$ are local real analytic/holomorphic generators of $\mathrm{im}(\dd)|_\mathcal{U}\subset \Gamma(E)|_\mathcal{U}$. We have $d\geq k$, where $k$ is the rank of $\dd$ on regular points. Let $M_{\mathrm{sing}}:=M\setminus M_{\mathrm{reg},\dd}$. If $\mathcal{U}\cap M_{\mathrm{sing}}$ is empty, then there is nothing to prove. Assume that $\mathcal{U}\cap M_{\mathrm{sing}}\neq \emptyset$. There exists real analytic/holomorphic functions $f_i^j\in\mathcal{O}_\mathcal U$ with $i=1,\ldots,d'$ and $j=1,\ldots,d$ such that $$\displaystyle{
\dd(u_j)=\sum_{i=1}^{d'}f_i^je_i}.$$
    
Now, consider the $d'\times d$-matrix $\mathfrak M=(f_i^j)$. The rank of $\mathfrak M$ is equal to  $r$ on $M_{\mathrm{reg},\dd}$ and is less than $r$  on $M_\mathrm{sing}$. Let $q=\mathrm{rk}(E)-k$. For $x\in M_{\mathrm{reg},\dd}$, denote by  $[\mathfrak M (x)]$ be the point $\mathrm{im}(\dd_x)\in \mathrm{Gr}_{-q}(E)$ generated by the columns of $\mathfrak{M}(x)$. In order to pick $d'\times k$-matrix that represents $\mathrm{im}(\dd_x)$,  consider the following indexing set

$$I_n=\left\{(n_1,\ldots,n_k)\in \mathbb N^{d'}\mid 1\leq n_1<\cdots<n_k\leq n\right\}.$$ We use $I_{d'}$ to pick $k$-rows and $I_d$ to pick $r$-columns. Pick  $\mathfrak a\in I_{d'}$ and $\mathfrak b\in I_d$ and consider the $d'\times k$-matrix $\mathfrak M_\mathfrak b=(f_i^j)_{i=1,\ldots, d', j\in \mathfrak b}$ and $\mathcal{I}_\mathfrak b$ the ideal generated by all the $k\times k$-minors of $\mathfrak M_ \mathfrak b$, i.e., the ideal generated by the  determinants $\Delta_{\mathfrak a \mathfrak b}:=\mathrm{det}(f_i^j)_{i\in \mathfrak a,j\in\mathfrak b }$ with $\mathfrak a\in I_{d'}$.
%\begin{itemize}
%\item [-] the ideal $\mathcal{I}$ generated by all the $k\times k$-minors of $\mathfrak M$,
    \item [] 
%\end{itemize}

Notice that the zero locus $Z(\mathcal I)$ of the ideal $\mathcal I$ is exactly $\mathcal{U}\cap M_{\mathrm{sing}}$. There exists $\mathfrak b\in I_d$ such that $\mathcal{I}_\mathfrak b\neq 0$, let us pick such a $\mathfrak b\in I_d$ and consider the analytic variety which is given by  the zero locus $Z(\mathcal{I}_\mathfrak b)$ of the ideal $\mathcal{I}_\mathfrak b$. We define the following maps

\begin{equation}
    N\colon \mathcal{U}\setminus Z(\mathcal{I}_\mathfrak b)\longrightarrow \mathcal{U}\times \mathrm{Gr}_{-q}( \mathbb K^{d'}),\; x\longmapsto \left(x, [\mathfrak M_\mathfrak b(x)]\right) 
\end{equation}
and 

\begin{equation}
    H\colon \mathcal{U}\setminus Z(\mathcal{I}_\mathfrak b)\longrightarrow \mathcal{U}\times \mathbb P^{\tau},\; x\longmapsto \left(x, [\Delta_{\mathfrak a_0 \mathfrak b}\colon\cdots\colon \Delta_{\mathfrak a_{\tau} \mathfrak b}]\right) 
\end{equation}
Here, $\tau=\begin{pmatrix}
    d'\\k
\end{pmatrix}-1$ and $\mathfrak a_1,\ldots, \mathfrak a_\tau\in I_{d'}$. 

\begin{itemize}
 \item [(a)] We have $\overline{N(\mathcal{U}\setminus Z(\mathcal{I}_\mathfrak b))}\simeq \overline{H(\mathcal{U}\setminus Z(\mathcal{I}_\mathfrak b))}$ : to see this, consider the Plücker embedding [Chapter 1, Section 5, \cite{Weyman_2003}] $${Pl}\colon \mathrm{Gr}_{-q}(\mathbb K^{d'})\hookrightarrow \mathbb P^\tau$$ and define the map $\mathrm{id}\times Pl \colon \mathcal{U}\times \mathrm{Gr}_{-q}(\mathbb K^{d'})\longrightarrow \mathcal{U}\times \mathbb P^\tau$. We have that $(\mathrm{id}, Pl)\circ N=H$. Therefore,  the closure of the image of $N$ and $H$ are isomorphic.
    \item [(b)] The maps $\sigma\colon M_{\mathrm{reg},\dd}\longrightarrow \mathrm{Gr}_{-q}(E),\, x\longmapsto \mathrm{im}(\dd_x)$ of Equation \eqref{eq:Nash-vb-map} and $N$ concide on $\mathcal{U}\setminus Z(\mathcal{I}_\mathfrak b)$. This implies that $\overline{\sigma(\mathcal{U}\setminus \mathcal{U}\cap M_{\mathrm{sing}})}=\overline{N(\mathcal{U}\setminus Z(\mathcal{I}_\mathfrak b))}$. Also, the closure $\overline{N(\mathcal{U}\setminus Z(\mathcal{I}_\mathfrak b))}$ is a monoidal transformation of $\mathcal{U}$ with center $Z(\mathcal{I}_\mathfrak b)$ and is an analytic variety.
\end{itemize}
Different open neighborhoods glue together to give an analytic variety by the universal property of monoidal transformations. This proves the statement.
%
%
%Item 2 is a consequence of the fact that a monoidal transformation with smooth center is smooth.
\end{proof}

\begin{remark}
There is another natural sheaf  $ \mathcal I_{\mathrm{im}(\dd)} \subset \mathcal O$  of ideals that lead to another monoidal transformation. Consider the sheaf of ideals of $ \mathcal O$  of all local functions which, near every point, are of the form $ \langle \dd (e_1 ) \wedge \dots \wedge\dd (e_k), \alpha \rangle$ for some local sections $ e_1, \dots, e_k \in \Gamma(F)$ and $ \alpha \in \Gamma(\wedge^k E^*) $. Equivalently, it can also be seen as the sheaf of ideals of $\mathcal O $ generated by all $ k \times k$ minors of the $ {\mathrm{rk}}(E) \times {\mathrm{rk}}(F)$-matrix that represents $\dd$ on given local trivializations of $F$ and $E$.  This second monoidal transformation $(\mathrm{Bl}_{\mathcal I_{\mathrm{im}(\dd)}}(M), \pi) $ and $(\tilde{M} ,\pi)$ enters into a commutative diagram as follows: \begin{equation}\label{dia:dottedarrow2}  \xymatrix{  \mathrm{Bl}_{\mathcal I_{\mathrm{im}(\dd)}}(M) \ar@{.>}[r]\ar[d]_\pi&\ar[d]^\pi \tilde{M}\\ M \ar[r]^{=}& M } .  \end{equation}  
This can be seen as follows. Let $ \mathcal U$ be a coordinate neighborhood on which $E $ and $F$ are trivial bundles. Let $ e_1, \dots, e_{\mathrm{rk}(E)}$ be a trivialization of $E$ over $ \mathcal U$. We say that a subset $w$ of $k$-elements $ i_1, \dots, i_k$  in $ \{1, \dots,{\mathrm{rk}(E)} \}$ is \emph{admissible} if the family $e_\bullet(w) :=  e_{i_1}, \dots, e_{i_k}$ is admissible.  
  Then the sheaf of ideals $ \mathcal I_\dd$ is generated by
   $$   \mathcal I_{\mathrm{im}(\dd)} = \sum_{w \in \mathrm{Adm}} \mathcal I_{\mathrm{im}(\dd(e_\bullet(w)))}  $$
   where $\mathrm{Adm}$ is the collection of all admissible subsets in $ \{1, \dots,{\mathrm{rk}(E)} \}$. Since $ \mathrm{Adm} $
 is not empty, we can select one, say $ w_0$, and we then have  in view of Equation \eqref{eq:changeadmissible}
 $$  \mathcal I_{\mathrm{im}(\dd)}  =  \mathcal K \,  \mathcal I_{\mathrm{im}(\dd(e_\bullet(w_0)))}  $$
 where $ \mathcal K \subset \hat{\mathcal O}(\mathcal U)$ is the sub-$ \mathcal O(\mathcal U)$-module generated by the functions $ \phi_{ e_\bullet ( w) / e_\bullet ( w_0) } $ defined as in \eqref{eq:changeadmissible}. The existence of the dotted arrow is then a consequence of the criteria of Moody.
 \end{remark}

\begin{remark}
    Notice that, using the notations of the proof of Proposition \ref{prop:thisisNash}, we have $\mathcal{ U}\cap M_{\mathrm{ sing}}\subseteq Z(\mathcal{I}_{{\mathrm{im}}(\dd(e_\bullet)}) $. 
    In general, there is no equality. That is, the zero locus of the ideal with respect to which one considers the monoidal transformation does not need to coincide with the singular
locus $M_{\mathrm{sing}}$ of $\dd$. To have equality, A. Sertöz \cite{Ali-Sinan} introduced a notion of good generators for $\mathrm{im}(\dd)\subseteq \Gamma(E)$ as follows: $\mathrm{im}(\dd)\subseteq \Gamma(E)$  admits a \emph{good system of generators} if for any $x\in M_{\mathrm{sing}}$ there exists an open neighborhood $\mathcal{U}$  of $x$ and  sections  $s_1, \ldots, s_k$ of $\mathrm{im}(\dd)$ such that

\begin{enumerate}
    
    \item  $s_1, \ldots, s_k$ span $\mathrm{im}(\dd)|_{\mathcal{U}\cap M_{\mathrm{reg},\dd}}$,
\item $s_1, \ldots, s_k$ are linearly dependent on $\mathcal{U}\cap M_{\mathrm{sing}}$.

\end{enumerate}

where $k$ is the rank of $\dd\colon F\to E$ on $M_{\mathrm{sing}}$. This family is in particular admissible. Also, such a family exists  when $\mathrm{im}(\dd)$ is a projective submodule of $\Gamma(E)$.
With such generators one has $$Z(\mathcal{I}_\mathfrak b)=\mathcal{ U}\cap M_{\mathrm{ sing}}$$ for every neighborhood  $\mathcal{U}$ of a singular point. Nevertheless, the ideal $\mathcal{I}_{M_{\mathrm{sing}}}$ of vanishing functions on $M_{\mathrm{sing}}$ does not have to be equal to $\mathcal{I}_\mathfrak b$ on $\mathcal{U}$ but to its radical by the Nullstellensatz theorem \cite{zbMATH00704831}. 
%
%
%
%
%the zero locus $Z(I_\mathfrak b)$ does not necessarily coincide with $\mathcal{U}\cap M_{\mathrm{sing}}$.  However, in case we have $Z(I_\mathfrak b)=\mathcal{U}\cap M_{\mathrm{sing}}$ for every $\mathcal{U}\subseteq M$, $
%
%In that case, $\widetilde M$ is a monoidal transformation with center $M_{\mathrm{sing}}$. Therefore,  the smoothness of $\widetilde M$ is reduced to the smoothness of the singular locus $M_{\mathrm{sing}}$. % The sheafification of the image $\mathrm{im}(\dd)$ in $\Gamma(E)$ is a finitely generated subsheaf of a locally free sheaf $\Gamma(E)$.  By the criterion of \cite{Ali-Sinan}, the Nash blowup $\widetilde M$ is a monoidal transformation with center $M_{\mathrm{sing}}$. The latter is smooth if $M_{\mathrm{sing}}$ is smooth.
\end{remark}

In the sequel, we will not require the Nash blowup space $\widetilde M$ to be smooth. The established properties of $\widetilde M$ are sufficient to state the results we need.

\subsection{Nash blowups of singular foliations: main constructions and results}\label{Nash-foliations} Let $\mathfrak F$ be  a locally  finitely generated $\mathcal{O}$-submodule of $\mathfrak{X}(M)$, i.e., $\mathfrak F$ is a sub-sheaf $\mathfrak{X}(M)$ such that every point of $M$ admits an open neighborhood $\mathcal{U}$ and a finite number of vector fields $X_1,\ldots, X_n \in \mathfrak{X}(\mathcal U)$ such that for all $\mathcal V\subseteq \mathcal{U}$,  $\mathfrak{F}|_\mathcal{\mathcal{V}}=\sum_{k=1}^nf_kX_k|_{\mathcal{V}}$ for some $f_k\in \mathcal{O}_{\mathcal{V}}$. We assume that there exists a \emph{geometric resolution}, i.e.,  a complex of vector bundles $(E, \dd,\rho)$ of finite length \begin{equation}
    \label{eq:resolutions2}
  \xymatrix{ 0\cdots \ar[r] & E_{-i-1} \ar[r]^{{\dd^{(i+1)}}} \ar[d] & 
     E_{ -i} \ar[r]^{{\dd^{(i)}}}
     \ar[d] & E_{-i+1} \ar[r] \ar[d] & \ar@{..}[r] & \ar[r]^{{\dd^{(2)}}}& E_{-1} \ar[r]^{\rho=\dd^{(1)}} \ar[d]& TM \ar[d] \\M \cdots
      \ar@{=}[r] & \ar@{=}[r] M  &  \ar@{=}[r] M 
      &  \ar@{=}[r] M  &\ar@{..}[r] & \ar@{=}[r]   &  \ar@{=}[r] M  & M}
    \end{equation}
such that $\rho(\Gamma(E_{-1}))=\mathfrak{F}$ and which is exact as in Equation \eqref{eq:exact}. In the smooth case, geometric resolutions exist on every relatively compact open subset of $M$ such that every point admits local coordinates on which the local generators $\mathfrak F$ are real analytic, see \cite{LLS} or \cite{LLL} Section 2.6. In the holomorphic case, the existence of a geometric resolution in a neighborhood of each point is a property of coherent sheaves, see \cite{LLL} Section 2.6.

For every $i\geq 0$, let $M_{\mathrm{reg}^i,\mathfrak{F}}$
%(as in Lemma \ref{lemma:M_reg})
be the open dense subset of $M$ made of all points $ m \in M$ such that the image $\mathrm{im}(\dd^{(i+1)})$ of  the vector bundle morphism $E_{-i-1}\stackrel{\dd^{(i+1)}}{\longrightarrow}E_{-i}$  is of constant rank on some neighborhood. 
For $i=0$, we define it to be  open dense subset of $M$ made of all points $ m \in M$ such that the vector bundle morphism $\rho\colon E_{-1}\to TM$ is of constant rank on some neighborhood.
To avoid having to distinguish this case, from now on, we set $E_0:=TM$ and $\dd^{(1)}=\rho$  by convention.

It deserves to be noticed that $m \in M_{\mathrm{reg}^i,\mathfrak{F}}$ if and only if $m$ admits a neighborhood on which $\mathrm{im}(\dd^{(i+1)})=\ker (\dd^{(i)})$. By Lemma \ref{lemma:M_reg}(b), we have $M_{\mathrm{reg}^1,\mathfrak F}=M_{\mathrm{reg}^0,\mathfrak F}$. Since any two geometric resolutions of some $\mathfrak F $ homotopy equivalent, and since this property is invariant under homotopy equivalence, the open dense subset $M_{\mathrm{reg}^i,\mathfrak{F}}$ does not depend on the choice of a geometric resolution of $\mathfrak F\subseteq \mathfrak X(M)$. Also, we will denote $M_{\mathrm{reg}^1, \mathfrak F}=M_{\mathrm{reg}^0,\mathfrak F}$ simply by $M_{\mathrm{reg}, \mathfrak F}$ which coincides with the open dense subset of regular points of the singular foliation $\mathfrak F$.  We have a sequence of inclusions of open dense subsets: 

$$   M_{\mathrm{reg},\mathfrak{F}} \subset M_{\mathrm{reg}^2,\mathfrak{F}}\subset M_{\mathrm{reg}^3,\mathfrak{F}}\subset \cdots \subset M$$

% $$     \cdots  \subset M_{\mathrm{reg}^3,\mathfrak{F}}  \subset M_{\mathrm{reg}^2,\mathfrak{F}}  \subset   M_{\mathrm{reg},\mathfrak{F}} \subset M . $$
 These points have the following characterizations:
\begin{enumerate}
\item[$i=0,1$ :] $m \in M_{\mathrm{reg},\mathfrak{F}}$ if and only if there is a neighborhood on which the distribution $$m'\mapsto T_{m'} \mathfrak F =\left\{ X_{|_{m'}} | X \in \mathfrak F\right\}\subset TM$$ has constant rank, i.e., is a regular foliation.
\item[$i=2$ :] $ m \in M_{\mathrm{reg}^2,\mathfrak{F}}$ if and only if there is a neighborhood on which $ \mathfrak F\subset \mathfrak X(M)$ is a free module over functions,
\item[$i=3$ :]  $ m \in M_{\mathrm{reg}^3,\mathfrak{F}}$ if and only if there is a neighborhood on which $ \mathfrak F$ admits a geometric resolution of length $2$,
\item[$ \vdots$\, \,]  
\item[$i=n+1 $ : ]  $ m \in M_{\mathrm{reg}^{n+1}, \mathfrak F}$ if and only if there is a neighborhood on which $ \mathfrak F$ admits a geometric resolution of length $n$.
\end{enumerate}

%for all $i\geq 1$. 

\subsubsection{The blowup spaces associated to a singular foliation}
The blowup spaces are constructed as follows. Let $(M, \mathfrak F)$ be a singular foliation and $(E, \dd, \rho)$ be a geometric resolution of $\mathfrak F$ as in Equation \eqref{eq:resolutions2}. For every $i\geq 0$, we apply the Nash construction to $\dd^{(i+1)}\colon E_{-i-1}\to E_{-i}$. By convention, for $i=0$, this means that we apply it to $\rho\colon E_{-1}\to TM$, because we set $ E_0=TM$, and $ \dd^{(1)}=\rho$. Let us recall this construction.
\begin{enumerate}
    \item[(a)]  Let $ \Pi_i\colon \mathrm{Gr}_{-r_i}(E_{-i})\longrightarrow M$ be the Grassmann bundle of $E_{-i}$ with $r_i$ is as in Lemma \ref{lemma:M_reg} (d). Consider the natural section of $\Pi_i$ on $ M_{\mathrm{reg}^{i}, \mathfrak  F}$ defined by : \begin{equation}
    \sigma_i\colon M_{\mathrm{reg}^{i}, \mathfrak  F}\longrightarrow \mathrm{Gr}_{-r_i}({E_{-i}}),\, x\longmapsto \mathrm{im}\left(\dd^{(i+1)}_x\right)
\end{equation}
\item[(b)] Let $\widetilde{M}_i:= \overline{\sigma_i(M_{\mathrm{reg}^{i}, \mathfrak  F})}$ be the closure of the image of $\sigma_i$ in $\mathrm{Gr}_{-r_i}(E_{-i})$.  Let $\pi_i\colon \widetilde{M}_i\longrightarrow M$ denote the restriction of $\Pi_i$ to  $\widetilde{M}_i$. 
\end{enumerate}If $(E,\dd, \rho)$ is of finite length, we also apply the Nash construction to  the vector bundle morphism $\dd=\oplus_{i\geq 2}\dd^{(i)}\colon \oplus_{i\geq 2}E_{-i}\to \oplus_{i\geq 1}E_{-i}$ by considering the section $$\sigma_{\infty}\colon M_{\mathrm{reg},\mathfrak  F}\longrightarrow\coprod_{x\in M}\prod_{ i\geq 1}\mathrm{Gr}_{-r_i}({E_{-i}}|_x),\;x\mapsto(\sigma_1(x),\sigma_2(x), \ldots, \sigma_i(x),\ldots \,)$$ and define $\widetilde{M}_{\infty} :=\overline{\sigma_{\infty}(M_{\mathrm{reg}, \mathfrak  F})}$ which comes with a natural map $\pi_\infty\colon \widetilde M_\infty\to M$.

\begin{remark}
    $\widetilde{M}_{\infty}$ should be understood as the tuples made of elements $V_1\in \mathrm{Gr}_{-r_1}({E_{-1}}|_x), \ldots, V_i\in \mathrm{Gr}_{-r_i}({E_{-i}}|_x),\ldots$ such that there exists $(x_n)\in M_{\mathrm{reg},\mathfrak F}^{\mathbb{N}}$\, {such that} $\mathrm{im}\left(\dd^{(i+1)}_{x_n}\right)\underset{n \to +\infty}{\longrightarrow} V_i\; as \; x_n\underset{n \to +\infty}{\longrightarrow}x$ for all $i\in \mathbb{N}$. It is important to notice that all the $V_i$'s are given by the same sequence $(x_n)\in M_{\mathrm{reg},\mathfrak F}^{\mathbb{N}}$. In particular, for every $i\geq 1$ there is a natural map

    \begin{equation*}  \xymatrix{  \widetilde M_\infty \ar@{.>}[r]\ar[d]_{\pi_\infty}&\ar[d]^{\pi_i} \widetilde M_i\\ M\ar[r]^{=}& M.}  \end{equation*}
\end{remark}

By Proposition \ref{prop:bi}, for each $i\geq 0$, the projection $\pi_i\colon \widetilde M_i\rightarrow M$ is invertible on the open dense subset $M_{\mathrm{reg}^{i}, \mathfrak{F}}$, it is proper and surjective. Moreover,  for each point $x\in M$ and for every $i\geq 0$, the fiber $\pi^{-1}_i(x)$ is non-empty. Also, $\pi^{-1}_\infty(x)$ is non-empty.

\begin{definition}
   For each $i\geq 1$, the space $\widetilde M_i$ together with the map $\pi_i\colon \widetilde M_i\to M$ is called the $i$-th \emph{blowup space} of $(M,\mathfrak F)$. Likewise, $\pi_\infty\colon\widetilde {M}_\infty\to M$ is called the \emph{last blowup space} of $(M,\mathfrak F)$. 
\end{definition}

\begin{remark}
    $\pi_0\colon \widetilde M_0\to M$ is the Nash blowup of the singular foliation $\mathfrak{F}$ in the sense of \cite{JPTS}. Also, $\pi_1\colon \widetilde M_1\to M$ is the blowup in the sense of \cite{Rossi} and of \cite{MohsenOmar}. While for $i\geq 2$, the $\pi_i\colon \widetilde M_i\to M$'s do not exist in literature as blowups of the singular foliation $\mathfrak F$ to our knowledge,  but they still can be seen as a class of Nash blowups in the sense of \cite{Ali-Sinan}.
\end{remark}

%From now on, most of the proofs are delayed to Section \ref{sec:Proofs}. 
As sets, $\widetilde M_i$, $\widetilde M_\infty$ do not need to be manifolds. They can be singular, see Section  \ref{sec:smooth}. 
\begin{proposition}\label{prop:locally-aff}
 Let $\mathfrak F$ be a holomorphic singular foliation or a  real analytic singular foliation. 
 Then, for every  $i\geq 0$ or $i=\infty$, $\widetilde M_i$ is an analytic variety.
 For $ M=\mathbb K^n$ and $ \mathfrak F$ a singular foliation with polynomial generators, it is even a quasi-projective variety. 
Moreover, it is obtained, in a neighborhood of every point though a monoidal transformation with respect to an ideal of the sheaf of functions.
 
The same results hold for $\widetilde M_\infty$.
\end{proposition}
\begin{proof}
In all the cases above, there exists a geometric resolution $(E, \dd, \rho)$ of $\mathfrak{F}$ of finite length by trivial vector bundles, \cite{LLS}. Moreover, $\dd$ and $\rho$ are given by holomorphic, real analytic or polynomials depending on the context. Proposition \ref{prop:thisisNash} applied to $\dd^{(i)}\colon E_{-i}\to E_{-i+1}$ or $\rho\colon E_{-1}\to TM$ implies that $\widetilde M_i$ is  a analytic subvariety of the Grassmann bundle, given by a monoidal transformation.
\end{proof}

The following assertion follows from the existence of homotopy equivalence between any two geometric resolutions.

\begin{theorem}\label{thm:independence}
Let $i\in \mathbb N_0$ or $i=\infty$.  Let $\mathfrak F$ be a singular foliation on $M$ that admits a geometric resolution. For any two geometric resolutions of $\mathfrak{F}$, the corresponding   $\widetilde M_i$ are canonically isomorphic. % For each $i\geq 1$, $\widetilde M_i$ does not depend on the choice of  a geometric resolution of $\mathfrak{F}$. The same is true  for $\widetilde{M}_\infty$.
\end{theorem}

Theorem \ref{thm:independence} may be seen as a consequence of Remark \ref{rmk:invariance} since for any two resolutions, the differential map $\dd^{(i)} $ differs by transformations as in that remark. However, we prefer to establish it through the following results.

\begin{proposition}\label{thm:Transformation} Assume that the sequence \eqref{eq:resolutions2} is a geometric resolution for $\mathfrak F$. For every $x\in M$, for every $i\geq 1$ and $V\in \pi^{-1}_i(x)$ one has, \begin{equation}
     \mathrm{im}(\dd^{(i+1)}_{x})\subseteq V \subseteq \ker(\dd^{(i)}_x).
 \end{equation}{In particular, for all $x\in M_{\mathrm{reg}^{i}, \mathfrak  F}$ and $i\geq 1$,\;  $\ker(\dd^{(i)}_x)=\mathrm{im}(\dd^{(i+1)}_x)=\pi_i^{-1}(x)$}.% is reduced to a point in $\mathrm{Gr}_{-r_i}(E_{-i})$.}
 \end{proposition}

Let us now equip the geometric resolution with an universal Lie $ \infty$-structure whose bracket, that we denote by $(\ell_k)_{k \geq 2}$, restrict to $k $-linear maps on $ \oplus_{i \geq 2}  E_{-i} + {\mathrm{ker}}(\rho)$ that we denote by $(\{\cdots \}_k)_{k \geq 1} $. 
 
 \begin{proposition}\label{thm:grass} Fix a geometric resolution $(E, \dd, \rho)$ of $\mathfrak F$ and a universal Lie $\infty$-algebroid $(E,(\ell_k)_{k\geq 1}, 
 \rho)$ of $\mathfrak
 F$. The following are satisfied:
\begin{enumerate}
\item For every $x\in M$ and $V\in \pi_1^{-1}(x)$, the $2$-ary bracket $\{\cdot\,,\cdot\,\}_2$ on $\ker \rho_x$ restricts to $V$.
%and the image of  $V$ in  $H^{-1}(\mathfrak{F},x)\simeq \mathfrak{g}_x$, is a Lie subalgebra of codimension $r-\dim(L_x)$, where $\dim (L_x)$ is the dimension the leaf through $x$.
\item For all $x\in M$, and $\left(V_{1}\subset E_{-1}|_x,\ldots, V_k\subset E_{-k}|_x,\ldots\right)\in \pi^{-1}_\infty(x)$, we have $\{V_i,V_j\}_2\subset V_{i+j-1}$ for every $i, j\in \mathbb N_0$.
   
\end{enumerate}

\end{proposition}
In particular, these two items have obvious consequences. 
Recall that for every $x \in M$, $H^{-1}(\mathfrak{F},x)\simeq \mathfrak{g}_x$ is the isotropy Lie algebra, and that $\oplus_{i\geq 1}H^{(-i)}(\mathfrak F,x)$ comes with a canonical natural graded Lie algebra structure (see discussion of Section \ref{sec:universal} (\ref{item:infty-isotropy})).

\begin{corollary}
 Under the assumption of Proposition \ref{thm:grass}, or every $x\in M$ and 
 \begin{enumerate}
     \item $V\in \pi_1^{-1}(x)$,  the image of  $V$ in  $H^{-1}(\mathfrak{F},x)\simeq \mathfrak{g}_x$, is a Lie subalgebra of codimension $r-\dim(L_x)$, where $\dim (L_x)$ is the dimension the leaf through $x$.  

     \item The image of $(V_1, \ldots, V_k,\ldots)\in \pi_\infty^{-1}(x)$ in $\oplus_{i\geq 1}H^{(-i)}(\mathfrak F,x)$ is a graded Lie subalgebra.
 \end{enumerate}
\end{corollary}

\begin{remark}
    The $3$-ary bracket $\{\cdot\,,\cdot\,,\cdot\}_3$ does not restrict to elements of $\pi_\infty^{-1}(x)$ for $x\in M$.
\end{remark}
The corollary below is a direct consequence of Proposition \ref{thm:Transformation}, and is another manner to state that $M_i$ does not depend on the geometric resolution.
\begin{corollary}\label{cor:independence}
There are inclusions
\begin{equation}
    \widetilde M_i\hookrightarrow \coprod_{x\in M}\mathrm{Gr}_{-\left(r_i-\mathrm{rk}\left(\dd^{(i)}_x\right)\right)}(H^{-i}(\mathfrak{F},{x}))\quad\text{and}\quad\widetilde M_\infty\hookrightarrow \coprod_{x\in M}\prod_{ i\geq 1}\mathrm{Gr}_{-\left(r_i-\mathrm{rk}\left(\dd^{(i)}_x\right)\right)}(H^{-i}(\mathfrak{F},{x})).  
\end{equation}
\end{corollary}
\begin{proof}
Let $x\in M$ and $i\geq 1$.
By Proposition \ref{thm:Transformation}, elements $V\in \pi_i^{-1}(x)$ satisfy the inclusions, $\mathrm{im}(\dd^{(i+1)}_{x})\subseteq V \subseteq \ker(\dd^{(i)}_x)$, they correspond injectively to a (unique) sub-vector space of codimension $r_i-\mathrm{rk}(\dd^{(i)})$ in $H^{-i}(\mathfrak{F},{x})$. In particular, this implies the existence
of an inclusion $\pi_i^{-1}(x)\hookrightarrow \mathrm{Gr}_{-\left(r_i-\mathrm{rk}(\dd^{(i)})\right)}(H^{-i}(\mathfrak{F},{x}))$.
\end{proof}

%Ga $\mathfrak g_x$ is the isotropy Lie algebra of the singular foliation $\mathfrak F$ at $x \in M$ and $L_x$ is the leaf that passes through $x$, and 
We denote by $\mathrm{GrLie}_{-\left(r-\dim(L_x)\right)}(\mathfrak{g}_x)$ the sub-Grassmannian of Lie subalgebras of $\mathfrak{g}_x$ of codimension $r-\dim(L_x)$.

\begin{corollary}\label{cor:Moshen}
    The image of the inclusion  $\displaystyle{\widetilde{M}_1\hookrightarrow \sqcup_{x\in M} \mathrm{GrLie}_{-\left(r-\dim(L_x)\right)}(\mathfrak{g}_x)}$ is  the blowup space of Omar Moshen \cite{MohsenOmar}. 
\end{corollary}
\begin{proof}
    %In the following example, we show that our $\widetilde M_1$ is the blowup construction $\mathrm{blup}(\mathfrak F)$ given by Mohsen in Section 2.2 of \cite{MohsenOmar}

Let $\mathfrak F$ be a singular foliation that admits a geometrical resolution $(E, \dd, \rho)$. For every $x\in M$, the fiber $\mathrm{blup}(\mathfrak F)_x$ of \cite{MohsenOmar} is constructed out of minimal generators $X_1, \ldots, X_d$ of $\mathfrak F$ in a neighborhood of $x$ as follows: for $y\in M_{\mathrm{reg}, \mathfrak F}$, let $\phi_y$ be the surjective linear map  defined by
\begin{equation}
    \phi_y\colon  \frac{\mathfrak F}{\mathcal I_x\mathfrak F} \longrightarrow T_y\mathfrak F,\; \phi_y ([X_i]_x) = X_i(y),\quad\text{for all}\quad i \in \{1,\ldots, d\},
\end{equation}
where $T_y\mathfrak F$ is the image of the evaluation map $e_y\colon \mathfrak F\longrightarrow T_yM$ at $y$. By definition, $\mathrm{blup}(\mathfrak F)_x$ is made of subspaces $V \subseteq\frac{\mathfrak F}{\mathcal I_x\mathfrak F}$ such that there
exists a sequence $x_n\in  M_{\mathrm{reg}, \mathfrak F}$ such that
\begin{equation}
    x_n \longrightarrow x,\, \phi_{x_n}^{-1}(0)\longrightarrow V\in \mathrm{Gr}_{-r}\left(\frac{\mathfrak F}{\mathcal I_x\mathfrak F}\right).
\end{equation}
We claim that for every $x\in M$, $\mathrm{blup}(\mathfrak F)_x\simeq \pi_1^{-1}(x)$. Indeed, we can assume that $(E, \dd, \rho)$ is a minimal geometric resolution at $x$ such that $\rho(e_i)=X_i$ for $i=1, \ldots,d$, where $(e_i)_{i=1, \ldots, d}$ is a local frame of $E_{-1}$. Since $\frac{\Gamma(E_{-1})}{\mathcal{I}_{x'}\Gamma(E_{-1})}\simeq {E_{-1}}|_{x'}$ for all $x'\in M$, the anchor map defines an isomorphism $\overline{\rho}_x\colon {E_{-1}|_x}\longrightarrow \frac{\mathfrak F}{\mathcal I_x\mathfrak F}$ such that the diagram 
\begin{equation}  
\xymatrix{{E_{-1}|_x}\ar[r]_{\simeq}^{\overline{\rho}_x}\ar[d]^{\rotatebox[origin=c]{90}{$\simeq$}}_{\kappa_y} &\frac{\mathfrak F}{\mathcal I_x\mathfrak F}\ar[d]^{\phi_y}\\{E_{-1}|_y}\ar[r]_{{\rho_y}} & T_y\mathfrak F}
\end{equation}
commutes. The claim follows.

\end{proof}

\subsubsection{Lift of the singular foliation $\mathfrak F$ to the blowup spaces $\left(\widetilde M_i\right)_{0\leq i\leq\infty}$}
Assume now that $\mathfrak{F}$ is a singular foliation and that Equation \eqref{eq:resolutions2} is a geometric resolution of $\mathfrak{F}$ of finite length. Notice that the fiber product $\bigtimes^{ i\geq 1}_M\mathrm{Gr}_{-r_i}({E_{-i}})$ is finite since $(E,\dd, \rho)$ is a geometric resolution of finite length. Hence, $\bigtimes^{ i\geq 1}_M\mathrm{Gr}_{-r_i}({E_{-i}})$ is a smooth manifold.
\begin{definition}
  Let $i\geq 0$. We say that $X\in \mathfrak F$ \emph{lifts} to  $\widetilde M_i\subset \mathrm{Gr}_{-r_i}(E_{-i})$, or $\widetilde M_\infty$, if there exists a vector field $\widetilde X\in \mathfrak X\left(\mathrm{Gr}_{-r_i}(E_{-i})\right)$ or $\mathfrak X\left(\bigtimes^{i\geq 1}_M\mathrm{Gr}_{-r_i}({E_{-i}})\right)$,  projectable to $X$ and  tangent to $\widetilde M_i$ in the sense of Section \ref{Nagano}(\ref{def:tangent}).  We denote by $\widetilde X_i$ or $\widetilde X_\infty$ the restriction of $\widetilde X$ to $\widetilde{M}_i$ or $\widetilde X_\infty$ respectively.

  We say that a $\mathfrak F$ \emph{lifts} to $\widetilde M_i$ if every vector field $X\in \mathfrak F$ lifts to $\widetilde M_i$.
\begin{remark}
$\widetilde X_i$ on $\pi_i^{-1}(M_{\mathrm{reg}^{i},\mathfrak F})$ is tangent in the usual sense to the submanifold and projects to $X$ through $\pi_i$. In particular, if a lift exists, {its restriction to $\pi^{-1}_i(M_{\mathrm{reg}^{i},\mathfrak F})$ is unique because $\pi_i\colon \pi_i^{-1}(M_{\mathrm{reg}^{i},\mathfrak F})\overset{\sim}{\longrightarrow} M_{\mathrm{reg}^{i},\mathfrak F}$. Since the other points of $\widetilde M_i$  are limits of elements of $\pi_i^{-1}(M_{\mathrm{reg}^{i},\mathfrak F})$, thus its restriction to $\widetilde M_i$ is unique}.      
\end{remark}
\end{definition}
\begin{theorem}\label{thm:trans2} Let $\mathfrak F$ be a singular foliation on $M$ that admits a geometric resolution. For every $i\geq 0$, the following items hold:
     \begin{enumerate}
         \item Every vector field $X\in \mathfrak{F}$ lifts to a unique vector field  $\widetilde X_i$ on $\widetilde M_i$,
         \item the map $X\in \mathfrak F \longrightarrow \widetilde X_i\in \mathfrak X(\widetilde M_i)$ does not depend on any choices. {In particular, it is a Lie algebra morphism.}%follows by unicity%
         \item The module $\widetilde{\mathfrak F}_i$ over functions on $\widetilde M_i$ generated by the $\widetilde X_i's$ for $X\in \mathfrak F$, is a singular foliation.
     \end{enumerate}
     The same holds for $\widetilde M_\infty$.
\end{theorem}
The following definition then makes sense:
\begin{definition}
   For each $i\geq 1$, the singular foliation $\widetilde{\mathfrak F}_i\subset\mathfrak X(\widetilde M_i)$ of Theorem \ref{thm:trans2} is called the $i$-th \emph{blowup} of $\mathfrak F$ on $\widetilde M_i$. Likewise, $\widetilde {\mathfrak F}_\infty$ is called the \emph{last blowup} of $\mathfrak F$ on $\widetilde M _\infty$. 
\end{definition}

\begin{remark}
    Although the closed subset $\widetilde M_i$ may have singularities, the singular foliation $\widetilde{\mathfrak F}_i\subset\mathfrak X(\widetilde M_i)$ possesses smooth leaves by Theorem \ref{thm:Sussmann_SP}.
\end{remark}

\noindent
\subsubsection{The blowup foliations and their Lie $k$-algebroids}

For $1\leq k<\infty$, let $\pi_k\colon \widetilde M_k\to M$ be the $k$-th blowup space of the singular foliation $(M,\mathfrak F)$ and $\widetilde{\mathfrak F}_k$ be the $k$-th blowup of $\mathfrak{F}$ on $\widetilde M_k$.  In the following, $\pi_k^*E_{-i}$ stands for the restriction to $\widetilde M_k$  of the pull-back vector bundle $\Pi^*E_{-i}\to \mathrm{Gr}_{-r_k}(E_{-k})$. The  pullback of a vector bundle map $\dd\colon E_{-i}\to E_{-i+1}$ on $\widetilde M_k$ shall be denoted by $\pi_k^*\dd$.

Here are the main results of this section.

\begin{theorem}\label{thm:projective1} Let $\mathfrak F\subseteq \mathfrak X(M)$ be a singular foliation on $M$ that admits universal Lie $\infty$-algebroid $\left(E,(\ell_k)_{k\geq 1},\rho\right)$ built on a geometric resolution $(E,\dd, \rho=\dd^{(1)})$.  For every $k\geq 1$, there exists a subvector bundle of $K\subset \pi_k^*E_{-k}$

\begin{enumerate}%this is part of the proof
    \item  with $\Gamma(K)\subseteq \ker \dd ^{(k)}$ and $K|_{\pi_k^{-1}(M_{\mathrm{reg}^k,\mathfrak F})}= \ker \dd ^{(k)}|_{M_{\mathrm{reg}^k,\mathfrak F}}$
    \item and a vector bundle morphism $\widetilde{\rho}\colon \pi_k^*E_{-1}\to T\widetilde M_k$ such that $\widetilde{\rho}(\Gamma(\pi_k^*E_{-1}))=\widetilde{\mathfrak F}_k$. 
\end{enumerate}
%on risque d'augementer les relations 
so that the complex of vector bundles $$  \xymatrix{0\ar[r]\ar[d] & \frac{\pi_k^*E_{-k}}{K} \ar[r]^{\overline{\pi_k^*\dd^{(k)}}} \ar[d] & 
     E_{ -k+1} \ar[r]^{{\pi_k^*\dd^{(k-1)}}}
     \ar[d] & E_{-i+1} \ar[r] \ar[d] & \ar@{..}[r] & \ar[r]^>>>>{{\pi_k^*\dd^{(2)}}}& \pi_k^*E_{-1} \ar[r]^{\widetilde\rho} \ar[d]& T\widetilde M_k \ar[d] \\ 
     \widetilde M_k \ar@{=}[r] & \ar@{=}[r] \widetilde M_k  &  \ar@{=}[r] \widetilde M_k 
      &  \ar@{=}[r] \widetilde M_k  &\ar@{..}[r] & \ar@{=}[r]   &  \ar@{=}[r] \widetilde M_k  & \widetilde M_k}
    $$ is exact in degree $k$ and  comes equipped with a “natural”  Lie $k$-algebroid structure\footnote{We also make sense of the notion of Lie $k$-algebroid on a closed subset $S\subseteq M$  similarly as in Section \ref{Nagano}(\ref{ref:LA}).}. Also, $\frac{\pi_k^*E_{-k}}{K}\to\widetilde M_k$ only depends on the image of $\dd^{(k)}$ in $\Gamma(E_{-k+1})$ not on $E_{-k}$.  Here, the bar  $\overline{\pi_k^*\dd^{(k)}}$ stands for the quotient of the map   ${\pi_k^*E_{-k}} \stackrel{{\pi_k^*\dd^{(k)}}}{\longrightarrow} 
     E_{ -k+1}$.

\end{theorem}
%tensor product is an right exact functor
Here is a remarkable fact for $k=1$.
\begin{corollary}\label{thm:projective}Let $\mathfrak F$ be a singular foliation on $M$ that admits a geometric resolution $(E,\dd, \rho)$. 

\begin{enumerate}
    \item The singular foliation 
$\widetilde{\mathfrak{F}}_1\subset \mathfrak X(\widetilde M_1)$ is Debord, i.e., it is the image of a Lie algebroid\footnote{Lie algebroids in the sense of Definition \ref{Nagano}\eqref{ref:LA} are Lie algebroids in the usual sense when $\widetilde M_1$ is smooth.} over $\widetilde M_1$ whose anchor map is injective on an open dense subset. \item This Lie algebroid is the Lie algebroid\footnote{If $\widetilde M_1$ is smooth, it is the Lie algebroid in the usual sense. Otherwise, it means that it is the differentiation of Mohsen's groupoid along the fibers of the source map.} of the groupoid of O. Mohsen \cite{MohsenOmar}.
\end{enumerate}
\end{corollary}

\begin{proof}
    Item 1 follows from Theorem \ref{thm:projective1}. For item 2, we also need Corollary \ref{cor:Moshen}, and a line by line comparison with \cite{MohsenOmar}. 
\end{proof}

In Corollary  \ref{thm:projective}, we do not need  the existence of geometric resolutions of $\mathfrak F$. Its proof only needs an almost Lie algebroid over $\mathfrak F$. In the smooth case, the latter always exists as long as $\mathfrak F$ is finitely generated, see Proposition 3.8. in \cite{LLS}. 

\begin{corollary}
    If $\widetilde M_1$ is smooth, then the Lie algebroid of $\widetilde {\mathfrak F}_1$ is integrable to a (Debord) Lie groupoid and the groupoid of O. Mohsen is a quotient of the latter.
\end{corollary}
\begin{proof}
    By \cite{Debord}, a Lie algebroid $A\to \widetilde M_1$ whose anchor is injective on an open dense subset is integrable to a Lie groupoid referred as the Debord groupoid. The Debord groupoid is universal among the integrations of $A$. Therefore, the groupoid of O. Mohsen is a quotient.
\end{proof}

\section{Proof of the main results}\label{sec:Proofs}

In this section, we prove the results of Section \ref{Nash-foliations} whose proofs were delayed.

%Proposition \ref{thm:Transformation}, Proposition \ref{thm:grass} and Theorem \ref{thm:trans2}.

\subsection{Proof of Propositions  \ref{thm:Transformation} and \ref{thm:grass}}

\begin{proof}[Proof (of Proposition \ref{thm:Transformation})]
We know by Proposition \ref{prop:bi}(2) that, for every $x\in M$ and $V\in \pi_i^{-1}(x)$, one has $\mathrm{im}(\dd^{(i+1)}_{x})\subseteq V$. Now,  for any element $v\in V$, there exists a sequence $v_n\in \ker(\dd^{(i)}_{x_n})=\mathrm{im}(\dd^{(i+1)}_{x_n}),\, n\in \mathbb{N}$ that converges to $v$. In particular, $\dd^{(i)}_{x_n}(v_n)=0$ for all $n$. Hence, by continuity, one has $v\in \ker(\dd^{(i)}_{x})$. Hence, $V\subseteq \ker \dd^{(i)}_x$. This completes the proof.

\end{proof}
\begin{proof}{(of Proposition \ref{thm:grass}).} For all $i\geq 1$, 
choose a local frame $e_1^{(i)},\ldots,e_{q_i}^{(i)},\ldots ,e_{q_i+r_i}^{(i)}$ of $E_{-i}$ on a neighborhood $\mathcal{U}$ of $x$ such that $e_1^{(i)}(x),\ldots,e_{q_i}^{(i)}(x)$ is an orthogonal basis for $V_{i}$ for an arbitrary Hermitian structure on $E_{-i}$. For $i, j\geq 1$, let $({c^{ij,s}_{kl}})\in \mathcal{O}_{\mathcal{U}}$ be a family of functions over $\mathcal{U}$ such that for all $k\leq q_i$ and $l\leq q_j$, $$\ell_2\left(e^{(i)}_{k}, e_{l}^{(j)}\right)=\sum_{s\geq 1}c_{kl}^{i j,s}e_s^{(i +j-1)}\in \Gamma_\mathcal{U}(E_{-i-j+1}).$$  In particular, \begin{equation}\label{eq:well-def}
    \left\{e^{(i)}_{k}(x), e_{l}^{(j)}(x)\right\}_2=\sum_{s\geq 1}c_{kl}^{ij,s}(x)e_s^{(i+j-1)}(x).
\end{equation}The bracket in Equation \ref{eq:well-def} is well-defined even for $i=1$ or $j=1$, although only the $2$-ary bracket of local sections is defined in such cases, because even if $i$ or $j=1$, we are taking the brackets of elements in $\ker\rho_x$. Let $u\in V_{i},v\in V_{j}$ with $\displaystyle{u=\sum_{s=1}^{q_{i}}\alpha^se_s^{(i)}(x)}$, and  $\displaystyle{v=\sum_{s=1}^{q_{j}}\beta^se_s^{(j)}(x)}$.\\
    
\noindent
Let $(x_n)\in M_{\mathrm{reg}^{i},\mathfrak F}^{\mathbb{N}}$ be a sequence of regular points that converges to $x$ such that $\mathrm{im}(\dd^{(i+1)}_{x_n})\underset{n \to +\infty}{\longrightarrow} V_i$ and $\mathrm{im}(\dd^{(j+1)}_{x_n})\underset{n \to +\infty}{\longrightarrow}V_j$. There exist sequences $$u_n=\displaystyle{\sum_{k=1}^{q_{i}+r_{i}}\alpha_{n}^ke_k^{(i)}(x_n)}\underset{n \to +\infty}{\longrightarrow} u;\quad v_n=\displaystyle{\sum_{l=1}^{q_{j}+r_{j}}\beta_{n}^le_l^{(j)}(x_n)}\underset{n \to +\infty}{\longrightarrow} v$$ with $u_n\in \mathrm{im}(\dd^{(i+1)}_{x_n})=\ker\dd^{(i)}_{x_n}$ and $v_n\in \mathrm{im}(\dd^{(j+1)}_{x_n})=\ker\dd^{(j)}_{x_n}$, for all $n\in \mathbb{N}$. In particular, the sequences $(\alpha^k_{n}),\,(\beta^l_{n})\in \mathbb{K}^\mathbb{N}$ satisfy $\alpha^k_{n}\underset{n \to +\infty}{\longrightarrow} \alpha^k;\quad \beta_{n}^l\underset{n \to +\infty}{\longrightarrow} \beta^l$ with $\alpha^k=\beta^l=0$ for $k\geq q_{i}+1,\,l\geq q_{j}+1$. Therefore, for every $n\in\mathbb{N}$ we have
\begin{align}
    \sum_{}^{}\alpha^{k}_{n}\beta^{l}_{n}c_{k l}^{i j,s}(x_n)e_s^{(i+j-1)}(x_n)=\label{eq:stability-bracket}\{u_n,v_n\}_{2}\in \mathrm{im}(\dd^{(i+ j)}_{x_n})=\ker\dd^{(i+j-1)}_{x_n}).
\end{align}
We have used in \eqref{eq:stability-bracket}, the fact that $\{\dd  u_1, \dd u_2\}_2\in \mathrm{im}(\dd),$ for all $u_1,u_2\in E_{\leq -2}$.
Since \begin{align}
    \nonumber\sum_{}^{}\alpha^{k}_{n}\beta^{s}_{n}c_{kl}^{i j,s}(x_n)e_s^{(i+j+1)}(x_n)\underset{n \to+\infty}{\longrightarrow}& \sum_{}^{}\alpha^{k}\beta^{l}c_{kl}^{i j,s}(x)e_s^{(i+j-1)}(x)\in {E_{-i-j+1}}|_x\\\label{eq:Lbracket}&=\{u,v\}_{2}.
\end{align}As a result,  $\{u,v\}_{2}\in V_{i+j-1}\in \pi_{i+j-1}^{-1}(x)$. Hence, for every point $(V_1,\ldots, V_i,\ldots, V_j, \ldots\;)\in 
~\pi_{\infty}^{-1}(x)$ one has $\{V_i,V_j\}_{2}\subseteq V_{i+j-1}$. This proves item 2. By taking $i=j=1$ and $V_i=V_j=V\in \pi^{-1}_1(x)$, Equation \eqref{eq:Lbracket} means that $\{u, v\}_2\in V$. This proves item 1.
 \end{proof}

\subsection{Proof of Theorem \ref{thm:independence}}

In this section, we give a second proof of Theorem \ref{thm:independence}, which is interesting by itself, because it uses a method that we will use in the subsequent proofs. By Corollary \ref{cor:independence} (whose
proof is independent of Theorem \ref{thm:independence}), for every $i\geq 1$, we have an inclusion  $\widetilde M_i\hookrightarrow \coprod_{x\in M}\mathrm{Gr}_{-\left(r_i-\mathrm{rk}\left(\dd^{(i)}_x\right)\right)}(H^{-i}(\mathfrak{F},{x}))$, where $r_i$ is defined as in Lemma \ref{lemma:M_reg}($d$).  We now need to show this inclusion is canonical, i.e., independent of the choice of a geometric resolution $(E, \dd, \rho)$.

\begin{convention}
For $(E,\dd,\rho)$  a geometric resolution of $\mathfrak F$.  Denote by $\pi_i^E\colon \widetilde M_i^E\to M$ the Nash blowup space  constructed out of a geometric resolution $(E, \dd, \rho)$ and $\pi_i^{E'}\colon \widetilde M_i^{E'}\to M$ the Nash blowup space constructed out of a geometric resolution $(E', \dd', \rho')$ for $i\geq 1$. Also, for $x\in M$ and  $V\in \pi^{-1}_i(x)$, we denote by $\overline{V}$ the image of $V$ in $\mathrm{Gr}_{-\left(r_i-\mathrm{rk}\left(\dd^{(i)}_x\right)\right)}(H^{-i}(\mathfrak{F},{x}))$.
\end{convention}
\begin{remark}
Let $x\in M$. Consider a minimal geometric resolution $(E', \dd', \rho')$ of $\mathfrak F$ at $x$ (see Definition (\ref{def: geom.resol})). For $V\in(\pi_1^E)^{-1}(x)$ and $V'\in(\pi_1^{E'})^{-1}(x)$ one has that $\dim V'\leq \dim V$, because $\mathrm{rk}(E_{-1}')\leq \mathrm{rk}(E_{-1})$ by minimality. Hence, $V, V'$ do not necessarily belong to the same Grassmannian. However,  $\dim \overline{V}=\dim \overline{V'}$. We prove the latter in the next Lemma. %In general, the latter equality fails when $i\geq 2.$
\end{remark}

%Let us first prove the following lemma.

\begin{lemma}\label{lemma:dimensions}
{Let $(E, \dd, \rho)$ and $(E', \dd', \rho')$ be geometric resolutions of $\mathfrak F$. For all $i\geq 1$,  and for all  $V\in(\pi_i^{E})^{-1}(x)$ and $V'\in(\pi_i^{E'})^{-1}(x)$, one has  $\dim \overline{V}=\dim \overline{V'}$.}
\end{lemma}
\begin{proof}
If $x\in M$ is a regular point, then $\overline{V}=\overline{V'}=\{0\}$. Thus, the equality holds. Let $x\in M$ be a singular point. We prove it only for  $i=1, 2$, since  $i=1$ is a special case and  for $i\geq 3$ the proof uses a similar argument as for the one of $i=2$. The key point in the latter is, for every $x\in M$, the restriction of the complexes $(E, \dd, \rho)$ and $(E', \dd', \rho')$ at $x$ are quasi-isomorphic. This implies that  the codimension of $\mathrm{im}\left(\dd_x^{(i+1)}\right)$ inside $\ker \dd_x^{(i)}$, resp. $\mathrm{im}\left({\dd_x'}^{(i+1)}\right)$ inside $\ker {\dd_x'}^{(i)}$, is invariant.

Let $V\in(\pi_1^{E})^{-1}(x)$ and $V'\in(\pi_1^{E'})^{-1}(x)$. We have
\begin{align*}
  \dim \overline{V}&=\dim V-\dim (\mathrm{im}\,(\dd_x^{(2)}))\\&=\dim V-(\dim \ker \rho_x-\dim\ker \rho'_x+\dim (\mathrm{im}\,({\dd_x'}^{(2)}))\\&= \dim V -\mathrm{rk}({E_{-1}})+\mathrm{rk}({E'_{-1}})-\dim (\mathrm{im}\,({\dd_x'}^{(2)}))\\&=\dim V'-\dim (\mathrm{im}\,({\dd'_x}^{(2)}))\\&=\dim\overline{V'}.
\end{align*}
We have used the fact the cohomology groups at degree $-1$ of both complexes are isomorphic and the Rank–nullity theorem.

For $i=2$, let $V\in(\pi_2^{E})^{-1}(x)$ and $V'\in(\pi_2^{E'})^{-1}(x)$. Notice that $\dim V=\mathrm{rk}(E_{-2})-\mathrm{rk}(E_{-1})+r$. We have a similar formula for $\dim V'$. By direct computation we find that
\begin{align}
 \nonumber\dim \overline{V}&= \dim V- \dim (\mathrm{im}\,\dd_x^{(3)})\\\label{eq:invariant}&=\dim V-\mathrm{rk}(E_{-2})+ \mathrm{rk}(E'_{-2})+ \dim (\mathrm{im}\,(\dd_x^{(2)}))- \dim (\mathrm{im}\,({\dd_x'}^{(2)}))- \dim (\mathrm{im}\,({\dd_x'}^{(3)})).
\end{align}
We have used the fact the cohomology groups at degree $-2$ of both complexes are isomorphic and the Rank–nullity theorem. But $$\dim (\mathrm{im}\,(\dd_x^{(2)}))=\mathrm{rk}(E_{-1})-\dim (\mathrm{im}(\rho_x))- \dim W,$$ where $W$ is such that $\dim (\mathrm{im}\,(\dd_x^{(2)}))\oplus W=\ker \rho_x$. A similar formula holds for $\dim (\mathrm{im}\,({\dd'_x}^{(2)}))$ by adding ${'}$ everywhere. Substituting them into the Equation \eqref{eq:invariant}  we obtain

$$\dim \overline{V}= \dim \overline {V'}+ \dim W'-\dim W= \dim \overline {V'},$$ since $\dim W'=\dim W$.

\end{proof}

\begin{proof}[Proof of (Theorem \ref{thm:independence})]
For simplicity, we prove it for $i=1$. For $i\geq 1$, the same arguments hold.

Let $(E, \dd, \rho)$ and $(E', \dd', \rho')$ be geometric resolutions of $\mathfrak F$. There exists chain morphisms $\varphi \colon E\rightarrow E'$ and $\psi \colon E'\rightarrow E $ whose compositions are homotopic to identity. In particular, $\varphi, \psi$ induce well-defined isomorphisms $\overline{\varphi}$ and  $\overline{\psi}$ at the level of cohomology which are inverse to each other. The latter is canonical, see \cite{LLS}, Lemma 4.1. All we need to show is $\overline{\varphi}$ sends $\widetilde M_1^{E}$ to $\widetilde M_1^{E'}$.

Let $x\in  M$. Let $e_1, \ldots, e_k$ %and $e'_1, \ldots, e'_{k'}$ 
be local sections around $x$ of $E_{-2}$ %and, $E_{-2}'$ respectively 
such that $$\mathrm{span}\left(\dd^{(2)}e_1|_x,\ldots, \dd^{(2)}e_k|_x \right)=\mathrm{im}(\dd^{(2)}_x).$$ There is a neighborhood  $U_x$ of $x$ such that $F_y:=\mathrm{span}\left(\dd^{(2)}e_1|_y,\ldots, \dd^{(2)}e_k|_y \right)\subseteq \mathrm{im}(\dd^{(2)}_y)$ with $y\in U_x$ is of constant rank. These sections define a vector bundle $F$ on $U_x$ and $F_x=\mathrm{im}(\dd^{(2)}_x)$. Likewise, by shrinking $U_x$ if necessary, one consider the vector bundle $ F'\subseteq \mathrm{im}({\dd'}^{(2)})$  on a neighborhood $U_x$ of $x$ such that $\varphi_y(F_y)\subseteq F'_y$. Therefore, for every $y\in U_x$,  $\varphi_y$ induces a map $\hat{\varphi}_y$  $$\frac{\mathrm{ker}(\rho_y)}{F_y}\longrightarrow \frac{\mathrm{ker}(\rho_y)}{F_y'}$$ which coincides with the isomorphism $\overline{\varphi}_x\colon \frac{\mathrm{ker}(\rho_x)}{\mathrm{im}(\dd^{(2)}_x)}\stackrel{\simeq}{\longrightarrow} \frac{\mathrm{ker}(\rho'_x)}{\mathrm{im}(\dd'^{(2)}_x)}$ at $x$. The map $\Hat{\varphi}$ induces a well-defined map from $\mathrm{Gr}_{-r}\left(\frac{E_{-1}}{F}\right)$ to $ \mathrm{Gr}_{-r}\left(\frac{E'_{-1}}{F'}\right)$ on a smaller open neighborhood $U_x$ of $x$.  Let $V\in(\pi_1^{E})^{-1}(x)$ and $V'\in(\pi_1^{E'})^{-1}(x)$ and let $(x_n)_{n\in \mathbb N}$ be a sequence of regular points in $U_x$ converging to $x$ such  that $\mathrm{im}(\dd^{(2)}_{x_n})=\ker\rho_{x_n}$ and $ \mathrm{im}(\dd'^{(2)}_{x_n})=\ker\rho'_{x_n}$ converge to $V$ and $V'$ respectively.

 This implies that the sequence  $\frac{\ker\rho_{x_n}}{F_{x_n}}$ converges to $[V]=\frac{V}{F_x}=\frac{V}{\mathrm{im}(\dd^{(2)}_x)}$ in $\mathrm{Gr}_{-r}\left(\frac{E_{-1}}{F}\right)$.  Since $\hat{\varphi}_{x_n}\left(\frac{\ker\rho_{x_n}}{F_{x_n}}\right)\subseteq \frac{\ker\rho'_{x_n}}{F'_{x_n}}$, it follows that $\hat{\varphi}_x([V_x])\subseteq [V'_x]$ where $V'_x$ is the limit of (a sub-sequence of) $\ker\rho'_{x_n}$. By  Lemma \ref{lemma:dimensions},  $[V]$ and $[V']$ have the same dimension, thus, $\overline{\varphi}_x([V])= [V']$.   Also,  $\overline{\psi_x}(\overline{V'})=\overline{V}$ since $\overline{\psi_x}$ and $\overline{\varphi_x}$ is are the inverse of each other. This defines the required map and completes the proof.

\end{proof}
\subsection{Proof of Theorem \ref{thm:trans2} and \ref{thm:projective1}}
Theorem \ref{thm:trans2} follows from Lemma \ref{prop:lift-on-blowup} which itself requires Lemma \ref{lemma:vf-Gr}. We prove those in the smooth context. Their proof are similar in the holomorphic context. We recall that for $p\colon E\longrightarrow M$ a vector bundle over $M$, a linear vector field on $E$ is a pair $(Z,X)\in \mathfrak{X}(E)\times \mathfrak{X}(M)$ such that
$$\xymatrix{E\ar[r]^Z \ar[d]_p&TE\ar[d]^{dp}\\M\ar[r]^X&TM}$$
is a morphism of vector bundles (see e.g \cite{Mackenzie-Kirill}, p. 110). Equivalently, 

\begin{enumerate}
    \item $Z[C_{\mathrm{lin}}^\infty(E)]\subset C_{\mathrm{lin}}^\infty(E)$ and $Z[p^*C^\infty(M)]\subset p^*C^\infty(M)$.
    \item[] or
    \item The flow of $Z$ on $E$ are (local) vector bundle isomorphisms $E\longrightarrow E$ over the flow of $X$ on $M$.
\end{enumerate}
where $C_{\mathrm{lin}}^\infty(E)$ is the subalgebra of smooth functions on $E$ which are fiberwise linear. The latter is canonically isomorphic to $\Gamma(E^*)$ as $C^\infty(M)$-modules. Notice in particular that, a linear vector field is $p$-projectable to $X$.
\begin{lemma}\label{lemma:vf-Gr}
  A linear vector field on $E\longrightarrow M$ induces a vector field on $\Pi\colon \mathrm{Gr}_{-q}(E)\longrightarrow M$ that is $\Pi$-projectable on $M$.  
\end{lemma}
\begin{proof}
Let $(Z,X)$ be a linear vector field on $E\longrightarrow M$. Its flow $\phi^{Z}_t\colon E\longrightarrow E$ is a vector bundle isomorphism over the flow $\phi^X_t\colon M\to M$ of $X$   whenever it is defined. Thus, %$\phi^{Z}_t$  induces a diffeomorphism $\mathrm{Gr}_{-q}(E)$ so that 
$\phi^{Z}_t$ induces a map $\mathrm{Gr}_{-q}(E)\longrightarrow \mathrm{Gr}_{-q}(E),\, V\mapsto \phi^{Z}_t(V)$ that we still denote by $\phi^{Z}_t$. 
Define ${\widetilde Z}\in \mathfrak X(\mathrm{Gr}_{-q}(E))$ for all $V\in \Pi^{-1}(x)$ by \begin{equation}
    \label{eq:pull-back-gr}{\widetilde Z}(V):=\frac{d}{dt}_{|_{t=0}}c(t)\in T_{V}\mathrm{Gr}_{-q}(E)
\end{equation}
so that  the flow  $\phi^{\widetilde Z}_t\colon \mathrm{Gr}_{-q}(E)\longrightarrow \mathrm{Gr}_{-q}(E)$  of $\widetilde Z$ at $V\in\Pi^{-1}(x)$ is  $\phi^{Z}_t|_x(V)$, where $c(t)=\phi^{Z}_t|_x(V)\in \Pi^{-1}(\phi^{X}_t(x))$ for $t$ in some interval $I$. Also, $\widetilde Z$ is $\Pi$-projectable to $X$, by construction. %because $\phi^{Z}_t\colon \mathrm{Gr}_{-q}(E)\longrightarrow \mathrm{Gr}_{-q}(E)$ is  .
\end{proof}
\begin{lemma}
 \label{prop:lift-on-blowup}
For every $X\in \mathfrak F$, there exists for all $i\geq 1$  a  linear vector field $(Z^i,X)$ on the vector bundle $p_i\colon E_{-i}\longrightarrow M$ and a linear vector field $(Z^0,X)$ on  $p_0\colon E_0:=TM\longrightarrow M$, $p_i$-projectable to $X$. Their flows are compatible with the
complex of vector bundles, 
\begin{equation}
    \quad\cdots \stackrel{\ell_1=\dd^{(4)}} \longrightarrow E_{-3} \stackrel{\ell_1=\dd^{(3)}}{\longrightarrow} E_{-2} \stackrel{\ell_1=\dd^{(2)}}{\longrightarrow} E_{-1} \stackrel{\rho=\dd^{(1)}}{\longrightarrow} TM.\end{equation}
i.e., 
%\begin{enumerate}
    %\item $Z^i$ is $p_i$-projectable to $X$,
%\item for  all  $f\in \mathcal{O}_M$,\; $\alpha\in \Gamma(E_{-i+1})$ and $e\in \Gamma(E_{-i})$,
    %\begin{equation}
        %\langle Z^1[Q^{(1)}f],e\rangle=\rho(e)[X[f]]\quad\text{and for  $i\geq 2$,}\quad \langle Z^i[Q^{(0)}\alpha],e\rangle=\langle Z^{i-1} [\alpha],{\dd^{(i)}(e)}\rangle
    %\end{equation}
%\item for $i\geq 0$, the flow  $\phi^{X^i}_t\colon E_{-i}\rightarrow {E_{-i}}$ of $X^i$ when it is defined, is a  vector bundle isomorphism over $\phi^{X^{i-1}}_t$,
%\item 
the diagram below commutes for all $i\geq 1$,
\begin{equation}\label{cube:diagram}
  \scalebox{0.7}{\hbox{
  \xymatrix{
& M \ar[rr]^{\phi^X_t} \ar@{<-}[dl]\ar@{=}'[dd]
& & M \ar@{=}[dd]\ar@{<-}[dl]
\\
E_{-i}\ar[rr]^>>>>>>>{\phi_t^{Z^i}} \ar[dd]_-{\dd^{(i)}}
& & E_{-i} \ar[dd]^<<<<<<<<<<<<<{\dd^{(i)}}
\\
& M \ar[rr]^>>>>>>>>>>>>>>>>>>>{\phi^X_t} \ar@{<-}[dl]
& & M \ar@{<-}[dl]
\\E_{-i+1}\ar[rr]_{\phi^{Z^{i-1}}_t} & & E_{-i+1}
}}}
\end{equation}

%\end{enumerate}

where $\phi^{Z^i}_t$ or $\,\phi^{X}_t$ denotes the flow of $Z^i$ or $X$, whenever defined. They induce  vector fields  ${\widetilde Z}^i$ on $\mathrm{Gr}_{-r_i}(E_{-i})$ such that
\begin{enumerate}
    \item ${\widetilde Z}^i$ is tangent to $\widetilde M_i$,
    \item ${\widetilde Z}^i$ projects onto $X$.
\end{enumerate}

 \end{lemma}
 \begin{proof}Consider $(E, \dd=\ell_1, \ell_2, \rho)$ the graded almost  Lie algebroid of $\mathfrak{F}$ induced by a universal Lie $\infty$-algebroid $(E, (\ell_k)_{k\geq 1},\rho)$, see Section \ref{def}(\ref{almost}).  Let $X\in \mathfrak{F}$ and $i\geq 0$. For $i\neq 0$, there exists a section $\upsilon$ of the vector bundle $p_i\colon E_{-i}\rightarrow M$ such that $\rho(\upsilon)=X$. Consider the linear vector field $Z^i\in \mathfrak{X}(E_{-i})$ defined as follows
 \begin{align}
    Z^i[p_i^*f]:&=p_i^*(X[f]),\;\forall\; f\in C^\infty(M),\\\label{eq:linear-vect}Z^i_e[\alpha]:&=X[\left\langle \alpha ,e\right\rangle]-\left\langle \alpha, \ell_{2}(\upsilon,e)\right\rangle,\;\forall\; \alpha\in \Gamma(E^*_{-i}),\;e\in\Gamma(E_{-i}).
 \end{align}
For $i=0$, one replaces $\ell_2(\upsilon, e)$ in \eqref{eq:linear-vect} by $[X, Y]$ with $Y\in \Gamma(E_0)=\mathfrak{X}(M)$.
Notice that $Z^i$ depends on the choice of the graded almost  Lie algebroid bracket $\ell_2$ and $X$. The fact that  Digram \eqref{cube:diagram} commutes follows the exact same lines of the proof given for Proposition 2.2.11 in a preprint version of \cite{LLL}, p. 99. Let us write it for the sake of completeness\footnote{Those arguments could be applied almost word for word to the holomorphic context.}. %To start with, notice that for $i \geq 1$,  we have the identities, \begin{align}\left\langle Z^{i}[Q^{(0)}\alpha],e\right\rangle= \left\langle Z^{i-1}[\alpha],\ell_1(e)\right\rangle,\quad \left\langle Z^{1}[Q^{(0)}\alpha],e\right\rangle= \left\langle Z^0[\alpha],\rho(e)\right\rangle.\end{align}for all $\alpha\in \Gamma(E_{-i}^*)$ and $e\in \Gamma(E_{-i})$.
By construction, the vector fields $(Z^i)_{i\geq 0}$  on  $E_{-i} \stackrel{p_i}{\to} M$  are $p_i$-related to $X$. This implies,  the vector field $(Z^{i},Z^{i-1})\in \mathfrak X (E_{-i} \times E_{-i+1}) $ is tangent to the fiber product $E_{-i} \times_{p_i,M,p_{i-1}} E_{-i+1} $. Thus, $(Z^{i},Z^{i-1})$ restricts to a linear vector field on $E_{-i} \oplus E_{-i+1} \stackrel{p}{\to} M$  denoted by $Z_{i,i-1}$. The latter is given by the formula \begin{align*}
    Z_{i,i-1}[p^*f]:&=p^*(X[f]),\;\forall\; f\in C^\infty(M),\\\label{eq:linear-vect}Z_{i,i-1}[\xi](e\oplus e'):&=\begin{cases}
        X[\left\langle \xi ,e\oplus e'\right\rangle]-\left\langle \xi, \ell_{2}(\upsilon,e)\oplus \ell_{2}(\upsilon,e') \right\rangle\quad \text{for}\; i\neq 1\\ X[\left\langle \xi ,e\oplus Y\right\rangle]-\left\langle \xi, \ell_{2}(\upsilon,e)\oplus [X,Y] \right\rangle\quad \text{for}\; i= 1, \; \text{and}\; e'=Y\in\mathfrak X(M)
    \end{cases}
 \end{align*}for all $\xi\in \Gamma\left((E_{-i}\oplus E_{-i+1}\right)^*_{-i}),\;e\in\Gamma(E_{-i}),\, e'\in\Gamma(E_{-i+1})$ and is again $p$-related to $X$. We now consider for $i\geq 1$ the graph $$\mathrm{Graph}(\dd^{(i)})=\displaystyle{\left\{\left(e, \dd^{(i)}(e)\right) \mid e \in E_{-i}\right\} \subset E_{-i} \oplus E_{-i+1}}$$ of $\dd^{(i)} \colon E_{-i}\to E_{-i+1}$ which is  submanifold of $E_{-i} \oplus E_{-i+1}$ with $\rho=\dd^{(1)}\colon E_{-1}\to E_0=TM$. Let us check that $Z_{i,i-1}$ is tangent to the submanifold $\mathrm{Graph}(\dd^{(i)})$:  This comes from the following items 
\begin{itemize}
     \item[-]  the submanifold $\mathrm{Graph}(\dd^{(i)})\subset E_{-i} \oplus E_{-i+1}$ is the zero locus of the ideal generated by the functions

 $$\begin{array}{rrcl} \xi_\alpha \colon & E_{-i} \oplus E_{-i+1} & \to &  \mathbb R \\& (e,e') & \mapsto & \left\langle\alpha, \dd^{(i)}(e)-e'\right\rangle  \end{array}
 $$
with $ \alpha \in \Gamma(E_{-i+1}^*)$.
     \item[-]  and the for all $(e,e')\in \mathrm{Graph}(\dd^{(i)})$:
 \begin{align*}\nonumber
     Z_{i,i-1}\, [\xi_\alpha ] \, (e,e')&=X[\left\langle \xi_\alpha ,e\oplus e'\right\rangle]-\left\langle \xi_\alpha, \ell_{2}(\upsilon,e)\oplus \ell_{2}(\upsilon,e') \right\rangle\\ &=\underbrace{X[\left\langle \alpha, \dd^{(i)}(e)-e'\right\rangle]}_{=0}-\left\langle \alpha, \dd^{(i)}\ell_2(\upsilon,e)-\ell_2(\upsilon,e'))\right\rangle\\\nonumber&=-\left\langle\alpha, \ell_2\left(\upsilon,\dd^{(i)}(e)\right)-\ell_2(\upsilon,e')\right\rangle\\\nonumber&=-\left\langle\alpha, \ell_2\left(\upsilon,\dd^{(i)}(e)-e'\right)\right\rangle=0.
     \end{align*}where we have used the compatibility condition of the $\dd=\ell_1,\ell_2$ brackets of the graded almost Lie algebroid.
 \end{itemize}
 
Now, let $\mathcal U,\mathcal{V}\subseteq M$ be the open subsets of $M$ and $I=(-\epsilon,\epsilon)\subseteq \mathbb R$ an interval of $\mathbb R$ where the flow $\phi_t^X\colon \mathcal{U}\to \mathcal{ V}$ of $X$ at time $t$ is defined. 
Recall that the flow  $\phi_t^{Z_i}\colon E_{-i}|_{\mathcal{U}} \to E_{-i+1}|_{\mathcal{V}}$ of $Z^i$ and the flow $\phi_t^{Z_{i-1}} \colon E_{-i+1}|_\mathcal{U} \to E_{-i+1}|_{\mathcal V
}$ of $Z^{i-1}$  are vector bundle isomorphisms  over $\phi_t^X\colon \mathcal{U}\to \mathcal{V}$.
The flow of $Z_{i,i-1}$ is also a vector bundle isomorphism given by the formula 
 $$  \begin{array}{rcl} (E_{-i} \oplus E_{-i+1})_{|_{\mathcal U}} & \to   &  (E_{-i} \oplus E_{-i+1})_{|_{\mathcal V}} \\  (e,e') & \mapsto  & \left(\phi_t^{Z^i}(e), \phi_t^{Z^{i-1}}(e')\right).\end{array}$$
Now, since $Z_{i,i-1}$ is tangent to $\mathrm{Graph}(\dd^{(i)})$, its flow preserves $\mathrm{Graph}(\dd^{(i)})$, that is, for all $e\in \Gamma(E_{-i})$ $$\left(\phi_t^{Z^i}(e), \phi_t^{Z^{i-1}}(\dd^{(i)}(e)\right)\in \mathrm{Graph}(\dd^{(i)}).$$ This implies that Diagram \eqref{cube:diagram} commutes. Therefore, the family $\left(\Phi^{Z^i}_t\right)_{i\geq 0}$ is an isomorphism of complex of vector bundles. This proves the first part of Lemma \ref{prop:lift-on-blowup}.\\

Now, by Lemma \ref{lemma:vf-Gr}, the linear vector field $(Z^i,X)$ induces a vector field  $\widetilde Z^i$ on the Grassmannian  bundle $\mathrm{Gr}_{-r_i}(E_{-i})$. Let us show item 1, $\phi^{Z^i}_t$ preserves $\widetilde M_i$ : to see this take $V\in\pi_i^{-1}(x)$,  let $x_n\underset{n \to +\infty}{\longrightarrow}x$ be such that $\mathrm{im}\,{\dd}^{(i+1)}_{x_n}\underset{n \to +\infty}{\longrightarrow} V$ with $(x_n)\subset{M_{\mathrm{reg}^{i},\mathfrak F}}$. Since $\dd^{(i+1)}\circ\phi_t^{Z^{i+1}}=\phi_t^{Z^{i}}\circ\dd^{(i+1)}$ for $i\geq 0$, one has  $$\phi_t^{Z^i}|_{x_n}\left(\mathrm{im}\,{ \dd}^{(i+1)}_{x_n}\right)=\mathrm{im}\,{\dd}^{(i+1)}_{\phi^X_t(x_n)},\quad \text{for every $n\in\mathbb{N}_0$}.$$ Thus, 
\begin{align*}
    \phi_t^{Z^i}|_x(V)&=\lim_{n\rightarrow +\infty}\phi_t^{Z^i}|_{x_n}\left(\mathrm{im}\,\dd^{(i+1)}_{x_n}\right)\\&=\lim_{n\rightarrow +\infty}\left(\mathrm{im}\,{\dd}^{(i+1)}_{\phi^X_t(x_n)}\right)\in \pi_i^{-1}\left(\phi^X_t(x)\right).
\end{align*}
Hence, the flow of $\widetilde Z_i$ preserves $\widetilde M_i$, i.e.,  ${\widetilde Z}^i$ is tangent to $\widetilde M_i$.%, by Equation \eqref{eq:pull-back-gr}.
\end{proof}
\begin{proof}[Proof (of Theorem \ref{thm:trans2}).]
By Lemma \ref{prop:lift-on-blowup}, every vector field $X\in \mathfrak F$ extends to a linear field $X^i\in \mathfrak X(\mathrm{Gr}_{-r_i}(E_{-i}))$ which is tangent to $\widetilde M_i$ in the sense of Definition \ref{Nagano}(\ref{def:tangent}). This proves item 1. 
Furthermore, the restriction $\widetilde X_i$ of $X^i$ to $\widetilde M_i$ is unique, since $\pi_i|_{\pi_i^{-1}(M_{\mathrm{reg}^{i}, \mathfrak F})}\colon\pi_i^{-1}(M_{\mathrm{reg}^{i}, \mathfrak F})\longrightarrow M_{\mathrm{reg}^{i}, \mathfrak F}$ is invertible. In particular, the map $X\in \mathfrak F \longrightarrow \widetilde X|_{\widetilde{M}_i}$ does not depend on any choices and is a Lie algebra morphism. The module which is generated by the $\widetilde X_i$ is closed under Lie bracket by item 2 of Theorem \ref{thm:trans2}). This ends the proof.
  %follows by unicity%
\end{proof}
\begin{proof}[Proof (of Theorem \ref{thm:projective1}).]Let $(E,\dd, \rho)$ be a geometric resolution of $\mathfrak F$. Fix a universal Lie $\infty$-algebroid of $\mathfrak F$ on $(E,\dd, \rho)$ and $k\geq 1$. Let $\tau^{E_{-k}}$ and $A^{E_{-k}}$ be the tautological subbundle and tautological quotient bundle on $\mathrm{Gr}_{-r}(E_{-k})$, that fit into the exact sequence\begin{equation}
     \xymatrix{&\tau^{E_{-k}}\ar@{^{(}->}[r]\ar[d]& \Pi_k^*E_{-k}\ar@{->>}[r]\ar[d]& A^{E_{-k}}\ar[d]& \\&\mathrm{Gr}_{-r_k}(E_{-k})\ar@{=}[r]&\mathrm{Gr}_{-r_k}(E_{-k})\ar@{=}[r]&\mathrm{Gr}_{-r_k}(E_{-k})&}
 \end{equation}
with  $A^{E_{-k}}\simeq \Pi^*E_{-k}/\tau^{E_{-k}}$. In particular, for $k=1$, $\mathrm{rk}(A^{E_{-1}})$ is the dimension of the regular leaves. One has
 \begin{enumerate}
 
     \item $\widetilde{\mathfrak F}_k$ the image of an almost Lie algebroid on $\Pi_k^*E_{-1}|_{\widetilde M_1}$ through the anchor map $$\widetilde\rho\colon\Gamma(\Pi_k^*E_{-1})|_{\widetilde M_k}\longrightarrow ~\mathfrak X(\widetilde M_k)$$ defined by $ \pi_k^*e\longmapsto~\widetilde{\rho(e)}\in \widetilde{\mathfrak F}_k$.
     \item  The tautological subbundle $\tau^{E_{-k}}$ lies in the kernel of the differential map $\dd^{(k)}\colon E_{-k}\to E_{-k+1}$: indeed, the fiber of $\tau^{E_{-k}}$ over a point $V\in \pi^{-1}_k(x)$ is equal to $V$ by definition. By Proposition \ref{thm:Transformation}, the latter is included in $\ker\dd^{(k)}_x$ with equality if  $x\in M_{\mathrm{reg}^{k}, \mathfrak F}$. Also, for $k=1$, $\tau^{E_{-1}}$ lies in the kernel of the anchor map $\rho=\dd^{(1)}$.
\end{enumerate}   
Therefore, the pull-back differential map $\pi_k^*\dd^{(k)}\colon \pi_k^*E_{-k}\to \pi_k^*E_{-k+1}$ goes to quotient to a well-defined vector bundle morphism $\pi_k^*\dd^{(k)}\colon \frac{\pi_k^*E_{-k}}{\tau^{E_{-k}}}\longrightarrow \pi_k^*E_{-k+1}$ which is injective on  the open dense subset $\pi_k^{-1}(M_{\mathrm{reg}^k,\mathfrak F})$ of $\widetilde M_k$. Denote by $K\to \widetilde M_k$ the restriction of $\tau^{E_{-k}}$ to $\widetilde M_k$. The $k$-th truncation of the pull-back of  the universal Lie $\infty$-algebroid of $\mathfrak F$ to $\widetilde M_k$ induces naturally a Lie $k$-algebroid on $$\frac{\pi_k^*E_{-k}}{K}\longrightarrow \pi_k^*E_{-k+1}\longrightarrow \cdots \longrightarrow \pi_k^*E_{-1}\longrightarrow T\widetilde M_k.$$ For $k=1$, the anchor map $\widetilde \rho$ goes to quotient
\begin{equation}
 \xymatrix{0\ar[r] &K\ar[r]&\pi_1^*E_{-1}\ar[r]\ar[d]^{\widetilde \rho} &A^{E_{-1}}|_{\widetilde M_k}\ar[r]\ar@{-->}[ld]&0\\& &T\widetilde M_1 & &.}
 \end{equation}and makes $\widetilde{\mathfrak F}_1$ the image of an almost Lie algebroid on $A^{E_{-1}}|_{\widetilde M_1}$ whose anchor is injective on the open dense subset $M_{\mathrm{reg},\mathfrak F}$. Thus, $A^{E_{-1}}|_{\widetilde M_1}$ is a Lie algebroid whose anchor is injective on $\pi_1^{-1}(M_{\mathrm{reg}, \mathfrak F})$, whose image is $\widetilde{\mathfrak F}_1$. This proves the result.
\end{proof}
\begin{remark}
    Notice that in the proof of Corollary \ref{thm:projective} we do not need the existence of a geometric resolution, we only make use of the anchor map and the bracket of an  almost Lie algebroid of $\mathfrak F$, i.e., we only need $E_{-1}$ and $\rho\colon E_{-1}\longrightarrow TM$.
\end{remark}

\section{Examples}\label{sec:Examples}
Let us start with some examples where our constructions give nothing new, i.e., $\widetilde M_i\simeq M$ or $\widetilde M_\infty\simeq M$.
\begin{example}\label{ex1}
If $\mathfrak{F}$ is a Debord singular foliation (i.e., $\mathfrak F$ is a projective submodule of $\mathfrak X(M)$), then $\widetilde M_i\simeq M$, for all $i\geq 1$ and $i=+\infty$. This comes from the fact that there exists a vector bundle $E_{-1}\longrightarrow M$ such that $\Gamma(E_{-1})\simeq \mathfrak{F}$ by Serre-Swan theorem \cite{SwanRichardG, MoryeArchanaS}. This isomorphism is given by a vector bundle morphism, $E_{-1}\stackrel{\rho}{\rightarrow}TM$ which is injective on the open dense subset $M_{\mathrm{reg}, \mathfrak  F}$. As a consequence,  $\cdots\rightarrow 0\rightarrow 0\rightarrow E_{-1}\stackrel{\rho}{\rightarrow}TM$ is a geometric resolution of $\mathfrak F$. Therefore, $\widetilde M_{i\geq 2}\simeq M$ since $E_{-i}=0$ for $i\geq 2$. Also, if $r$ is the dimension of the regular leaves of $\mathfrak F$, then  $r=\mathrm{rk}(E_{-1})$. Hence $\mathrm{Gr}_{-r}(E_{-1})\simeq M$. In particular, $\widetilde M_1\simeq M$.\end{example}

\begin{example}\label{ex2}
      If the regular leaves of $\mathfrak{F}$ are  open, then ${\widetilde M}_0\simeq M$, since $\mathrm{Gr}_{-0}(TM)\simeq M$. For instance, this happens for $\mathfrak F$  the singular foliation on $\mathbb R^N$ of vector fields vanishing at zero.
 \end{example}

 \begin{example}\label{ex3}
    If there exists a geometrical resolution $(E, \dd, \rho)$ of length $k$, then $\widetilde M_i\simeq M$ for all $i\geq k+1$. Notice that one also has $\widetilde M_k\simeq M$ since the last differential map $\dd^{(k)}\colon E_{-k}\longrightarrow E_{-k+1}$ is injective on an open dense subset so that the considered Grassmann bundle is $\mathrm{Gr}_{-\mathrm{rk}(E_{-k})}(E_{-k})\simeq~M$.

    %$\pi^{-1}_k(x)=\{o_x\}$ for all $x$...
\end{example}

 %\begin{example}Let $\mathfrak F$ be the vector field tangent to a submanifold $\Sigma$ of  $\mathbb K^N$.  Here $M_{\mathrm{reg},\mathfrak F}=\Sigma$.\end{example}

 In  contrast with Examples  \ref{ex1}, \ref{ex2} and \ref{ex3}, we have other  examples  where our construction is not trivial.

 \begin{example}Let $(M,\mathfrak F)$ be a singular foliation admitting a geometric resolution of length 2

$$\cdots \longrightarrow 0\longrightarrow 0\longrightarrow E_{-2}\stackrel{\dd^{(2)}}{\longrightarrow}E_{-1}\stackrel{\rho}{\longrightarrow}TM.$$ Here, $\mathrm{im}(\dd^{(2)})|_{M_{\mathrm{reg},\mathfrak F}}$ is a vector bundle of rank $\mathrm{rk}(E_{-2})$. On the open dense subset of regular points $M_{\mathrm{reg},\mathfrak F}$, the map  $\dd^{(2)}\colon E_{-2}\to E_{-1}$ is injective, and on $M_{\mathrm{sing}}=M\setminus M_{\mathrm{reg},\mathfrak F}$ it is not. For simplicity, assume that $M=\mathbb R^N$ or $\mathbb C^N$ and that the vector bundles $E_{-2}, E_{-1}$ are trivial so that $\dd^{(2)}$ becomes a $\mathrm{rk}(E_{-1})\times \mathrm{rk}(E_{-2})$-matrix with coefficient in the algebra of functions on $M$. %Following the notations of the proof of Lemma \ref{lemma:Nash-monoidal}, 
The zero locus of the ideal $\mathcal{I}_\mathfrak b$ generated by the minors of this matrix in a basis, is exactly $M_{\mathrm{sing}}$. By construction, the Nash blowup $\widetilde M_1$ is the blowup of $M$ along the ideal $\mathcal{I}_\mathfrak b$.

For instance, for $M=\mathfrak{gl}_d(\mathbb K)$ is the vector space of $d\times d$-matrix with coefficient in $\mathbb K=\mathbb R, \mathbb C$. Let $(M,\mathfrak F)$ be the singular foliations given by  the adjoint action of $\mathfrak{gl}_d(\mathbb K)$ on $\mathfrak{gl}_d(\mathbb K)$, that is 
$$
\mathrm{ad}(x)y=[x, y], \quad x, y\in \mathfrak{gl}_d(\mathbb K).
$$$\mathfrak F$ admits a geometric resolution of length $2$ (see Example 3.32 in \cite{LLS}) with

$$M\times \mathbb K^d\stackrel{\dd^{(2)}}{\longrightarrow} M\times \mathfrak{gl}_d(\mathbb K),\, (x, (\lambda_0,\ldots, \lambda_{d-1}))\mapsto (x, \sum_{i=0}^{d-1}\lambda_ix^i)$$
    and $$M\times \mathfrak{gl}_d(\mathbb K)\stackrel{\rho}{\longrightarrow}TM\simeq M\times\mathfrak{gl}_d(\mathbb K) ,\, (x, v)\mapsto (x, [x, v]).$$ The open dense subset of regular points of $(M,\mathfrak F)$ is the set of matrices $x\in M$ whose centralizer $C(x):=\ker\rho_x$ is of minimal dimension equal to $N$. Equivalently, $M_{\mathrm{reg},\mathfrak F}$ is made of the matrices $x\in M$ whose characteristic polynomial 
equals to the minimal polynomial, also known as non-derogatory matrices \cite{wang2019explicitdescriptioncentralizersmatrix}. For $d=2$, $\widetilde M_1\simeq\mathrm{Bl}_{\mathcal{I}_\mathfrak b}(\mathbb K^4)$ is the usual blowup of $\mathbb K^4$ along the ideal $\mathcal I_\mathfrak b$ generated  by $\{x_1-x_4, x_2, x_3\}$, which is smooth.

For $d\geq 3$, computations becomes  complicated, and  the singular locus is a cone.

\begin{example}
    The Nash blowup can be smooth, even if the singular locus is not. In the case of the adjoint action of $\mathfrak{su}(n)$, the singular locus is not smooth, but the blowup is smooth, see Example 3.11\cite{louis2024nash}.
\end{example}
 \end{example}
\begin{example}
   Consider the projective singular foliation $\mathfrak F$  on  $M=\mathbb C^N$ generated by the Euler vector field $\overrightarrow{E}=\sum_{i=1}^Nx_i\frac{\partial}{\partial x_i}$. Here, $M_{\mathrm{reg},\mathfrak F}=\mathbb C^N\setminus \{0\}$. It is easily checked that $\widetilde M_0$ is the closure of the graph $\{(x,[x_1:\cdots: x_N])\in \mathbb C^N\times \mathbb P^{N-1}(\mathbb C)\,\mid\,  x\neq 0\}$. The latter is the blowup of $\mathbb C^N$ at $0$. This is an example where $\mathfrak F$ is Debord and $\widetilde M_0\neq M$. In particular, by Example \ref{ex1}, $\widetilde M_0\neq \widetilde M_1= M$.
\end{example}
\begin{example}Let $\mathfrak F$ be the  singular foliation of all vector fields vanishing at the origin $0\in M=\mathbb C^N$. Here, $M_{\mathrm{reg}, \mathfrak F}=\mathbb C^N\setminus {\{0\}}$. Let us compute $\widetilde M_1$. A geometric resolution $(E,\dd, \rho)$ of $\mathfrak F$ is given in  Example 3.34 of \cite{LLS}. Here $E_{-1}\simeq \mathbb C^N\times \mathfrak{gl}_N(\mathbb C)$ and the anchor map $\rho$ is $E_{ij}\mapsto x_i\frac{\partial }{\partial x_j}$, where $\mathfrak{gl}_N(\mathbb C)$ is the vector space of $N\times N$ matrix with coefficient in $\mathbb C$ and $(E_{ij})_{i,j=1,\ldots, N}$ its canonical basis.

A direct computation for every $x\neq 0$ tells that $\ker \rho_x$ is the subspace of matrices $M\in \mathfrak{gl}_N(\mathbb C)$ such that $Mx=0$, where $x=(x_1,\ldots, x_N)$ is seen as a column vector. Equivalently, this kernel can be described as $N$ copies of $[x_1:\cdots: x_N]^{\perp}$. Hence, $\widetilde M_1$ is the blowup of $\mathbb C^N$ at the origin. This is an example of a
singular foliation whose regular leaves are open, but such that $\widetilde M_1\neq M$. In particular,  by Example \ref{ex2}, $\widetilde M_0\neq \widetilde M_1$. 
\end{example}
Here is an example related to Poisson manifolds.
\begin{example}
   Let $(M,P)$ a smooth or holomorphic Poisson manifold with $P\in \Gamma(\wedge^2TM)$. Consider the singular foliation  generated by the Hamiltonian vector fields associated to $P$, i.e., $\mathfrak F=P^{\sharp}(\Gamma(T^*M))$, where $P^\sharp\colon T^*M\longrightarrow TM,\; \alpha\mapsto P(\alpha, \,\cdot\,)$. Assume that a geometric resolution exists.  By Lemma \ref{prop:lift-on-blowup}, every Hamiltonian vector field lifts to a vector field tangent to $\widetilde M_i$, $i\geq 1$. It is natural to ask whether the Poisson bivector field $P$ lifts to $\widetilde M_i$. Assume that $\widetilde M_i$ is smooth. Since for every $i\geq 1$,  $\pi_i^{-1}(M_{\mathrm{reg}^i, \mathfrak F})\longrightarrow M_{\mathrm{reg}^i,\mathfrak F}$ is invertible, the restriction $P|_U$ lifts to a Poisson bivector field on $\pi_i^{-1}(M_{\mathrm{reg}^i, \mathfrak F})$. However, it  does not lift to $\widetilde M_i$ in general, even when $\widetilde M_i$ is smooth. Indeed, consider the Poisson manifold $M=\mathfrak{so}^*(3)\simeq \mathbb R^3$ with \begin{equation}\label{eq:Poisson1}
       P=x \frac{\partial}{\partial y}\wedge \frac{\partial}{\partial z}+ y\frac{\partial}{\partial z}\wedge \frac{\partial}{\partial x}+z\frac{\partial}{\partial x}\wedge \frac{\partial}{\partial y}.
   \end{equation} Here $\mathfrak F$ is generated by the vector fields $P^\sharp(dx)= z\frac{\partial}{\partial y}-y\frac{\partial}{\partial z}$,\, $P^\sharp(dy)= x\frac{\partial}{\partial z}-z\frac{\partial}{\partial x}$,\, $P^\sharp(dz)= y\frac{\partial}{\partial x}-x\frac{\partial}{\partial y}$. Let us compute $\widetilde M_1$. Given a point $m\in M_{\mathrm {reg}, \mathfrak F}=\mathbb R^3\setminus \{0\}$, we find that
   $$\ker P^\sharp|_m=\left\{(a,b,c)\in \mathbb R^3 \,\mid\,(a,b,c)\in [x(m):y(m):z(m)]\in \mathbb P^2(\mathbb R)\right\}=[x(m):y(m):z(m)]
   %\left\{(a,b,c)\in \mathbb C^3\, \mid\, xc=az=0;\, yc=bz;\,xb=ya  \right\}=\left\{\right\} 
   .$$ Hence, $\widetilde M_1$ is the usual blowup $\mathrm{Bl}_0(\mathbb R^3)$ of $\mathbb R^3$ at the origin.

   The bivector field $P$ does not lift to $\widetilde M_1$. Recall that the blowup of $\mathbb R^3$ at the origin   $\mathrm{Bl}_0(\mathbb R^3)\subset~ \mathbb R^3\times~\mathbb P^2$ is covered by three charts given by $x\neq 0$, $y\neq 0$ and $z\neq 0$.  Let us look at the $x$-chart where the projection $\pi_1$ becomes $(x,y,z)\mapsto(x,xy, xz)$. In this chart  $P$ pulls back to
   \begin{equation}\label{eq:counter-poisson}
       y\frac{\partial}{\partial z}\wedge \frac{\partial}{\partial x}+z\frac{\partial}{\partial x}\wedge \frac{\partial}{\partial y}+\frac{1}{x}(1+y^2+z^2)\frac{\partial}{\partial y}\wedge \frac{\partial}{\partial x}.
   \end{equation} For $x= 0$, Equation \eqref{eq:counter-poisson} is not defined. In conclusion, the Hamiltonian vector fields of the Poisson structure $P$ in \eqref{eq:Poisson1} lift to $\widetilde M_1$, but the bivector field $P$ does not lift to $\widetilde M_1$,  although $\widetilde M_1$ is smooth.
   \end{example}

   \begin{example}
Let $(E_{-1}, \lb, \rho)$ be a Lie algebroid over a manifold $M$ and  denote by $\mathfrak F= \rho(\Gamma(E_{-1}))$ the induced singular foliation. Assume there exists geometric resolutions for $\mathfrak F$. The Lie algebroid $E_{-1}$ acts on  the spaces $\widetilde M_i$ for all $i\in \mathbb N_0$, and also on $\widetilde M_\infty$, and \begin{equation}
    \xymatrix{&\mathfrak{X}(\widetilde M_i)\ar@{<-}[d] \\\Gamma(E_{-1})\ar@{-->}[ru]^{\overline{\rho}}\ar[r]_{\rho}& \mathfrak{X}(M)}
\end{equation} is a commutative diagram of Lie algebra morphisms, where $\Bar{\rho}$ is defined on a local frame $(e_k)_k$ of $E_{-1}$ by $e_k\mapsto \widetilde {\rho(e_k)}_i$. Here\, $\widetilde \cdot$\, is as in Theorem \ref{thm:trans2}. In addition, for each $i\in \mathbb N_0$, $\widetilde {\mathfrak F}_i$  is the image of a Lie algebroid on $\widetilde M_i$, namely the natural  pull-back  of the Lie algebroid $E_{-1}$ to $\widetilde M_i$. {In particular, if  $\mathfrak F$ is given by a Lie algebra action of a Lie algebra $\mathfrak g$ on $M$, then $\widetilde{\mathfrak F}_i$ is given by an action of $\mathfrak g$ on $\widetilde M_i$.}
   \end{example}

Let us now study some examples related to the notion of an affine variety in $\mathbb C^d$.\\

Let $\mathbb{A}^d$ be an affine space over $\mathbb{K}=\mathbb R$ or $ \mathbb C$ with a set of coordinates $x_1,\ldots, x_d$. Recall that an \emph{affine} variety $W$ is a subset of the affine space $\mathbb{A}^d$ given by the zero locus $Z(\mathcal{I}_W)$ of a radical ideal $\mathcal{I}_W\subseteq\mathbb{K}[x_1,\ldots, x_d]$ and equipped with the induced Zariski topology of $\mathbb{A}^d$. The \emph{coordinate ring} of $W$ is the quotient ring  $\mathcal O_W=\mathbb{K}[x_1,\ldots, x_d]/\mathcal{I}_W$. The Lie algebra $\mathfrak{X}(W)$ of \emph{vector fields} on $W$ are  derivations of $\mathcal O_W$. %If $(y_1,\ldots, y_d)$ is a regular set of generators of $\mathbb{K}\mathbb[W]$, then $\mathfrak{X}(W)\simeq \oplus_{j=1}^d \mathbb{K}\mathbb[W]\,\frac{\partial}{\partial y_j}$.
We denote by $W_{\mathrm {reg}}$ the set of regular points  of $W$. For every $x\in \mathbb A^d$ we denote by $\mathfrak m_x$ the maximal ideal of vanishing polynomials at $x$. See for instance, \cite{Hartshorne} for more details on these notions.

\begin{example}\label{ex:along-phi}
     Let $M=\mathbb C^d$ and $\varphi\in\mathbb{C}[x_1,\ldots,x_d]$. Consider the singular foliation $\mathfrak{F}_\varphi=\{X\in \mathfrak{X}(\mathbb{C}^d)\mid X[\varphi]=0\}$. In this case, $M_{\mathrm{reg},\mathfrak F_\varphi}=\{x\in \mathbb C^d, \mid, d_x\varphi\neq 0\}$. 
    For every $y\in \mathbb C^d$, $(T_y\mathfrak{F}_\varphi)^\bot= \langle \nabla_y\varphi\rangle$. For a convergent sequence $y_n\underset{n \to +\infty}{\longrightarrow}y$ with $y_n\in M_{\mathrm{reg}, \mathfrak {F}_\varphi}$. The sequence $\mathrm{im}(\rho_{y_n})=T_{y_n}\mathfrak F_\varphi$ converges if and only if $\nabla_{y_n}\varphi$ converges in $\mathrm{Gr}_{-(d-1)}(\mathbb{C}^d)$, that is, $\left[\frac{\partial \varphi}{\partial x_1}(y_n)\colon \cdots \colon\frac{\partial \varphi}{\partial x_d}(y_n)\right]$ converges in the projective space $\mathbb{P}^{d-1}(\mathbb{C})$. Therefore, $\widetilde M_0$ is the closure of the image of the map, $y\mapsto (y, \left[\frac{\partial \varphi}{\partial x_1}(y)\colon \cdots \colon\frac{\partial \varphi}{\partial x_d}(y)\right])$ which is the  blow up of $\mathbb{C}^d$ along the singular locus of $\varphi$, i.e., along the ideal generated by the components of  $d\varphi$. For instance, 

    \begin{enumerate}
        \item For $\varphi(x_1,\ldots, x_d)= \sum_{i=1}^dx_i^2$, $\widetilde M_0$ is the blowup of $\mathbb C^d$ along the ideal $(x_1,\ldots, x_d)$, i.e., the blowup of $\mathbb C^d$ at zero, which is smooth.

        \item For $\varphi(x_1,\ldots, x_d)= \sum_{i=1}^dx_i^3$, $\widetilde M_0$ is the blowup of $\mathbb C^d$ along the ideal $(x_1^2,\ldots, x_d^2)$. This is not the blowup of $\mathbb C^d$ at zero, and it is easily seen in the charts that is not smooth with a singularity at the origin.
    \end{enumerate}
    However, since the ideals $\langle x_1,\ldots,x_d\rangle$ and $\langle x_1^2,\ldots,x_d^2\rangle$ are related by $$\langle x_1,\ldots,x_d\rangle^{d-1} \langle x_1^2,\ldots,x_d^2\rangle=\langle x_1,\ldots,x_d\rangle^{d+1}$$ and since the blowup of $\mathbb C^d$ along the ideals $\langle x_1,\ldots,x_d\rangle$ and $\langle x_1,\ldots,x_d\rangle^{d+1}$ are the same, there is a map \begin{equation*}  \xymatrix{\mathrm{Bl}_{\langle x_1,\ldots,x_d\rangle}(\mathbb C^d) \ar@{.>}[r]\ar[d]_\pi&\ar[d]^\pi \widetilde M_0=\mathrm{Bl}_{\langle x_1^2,\ldots,x_d^2\rangle}(\mathbb C^d)\\ \mathbb C^d \ar[r]^{=}& \mathbb C^d}  \end{equation*}by Moody's criteria, see Section \ref{sec:smooth}.
     \end{example}

\begin{example}(Nash modification). \label{ex:Nash}
Let $M=W$ be an affine irreducible affine variety of dimension $r$ embedded in $\mathbb C^d$. Let  $\Sigma$ be its singular locus. Let $\mathfrak{F}= \mathrm {Der}(\mathcal O_W) %\{X\in \mathfrak{X}(W) \mid X[\mathcal{I}_\Sigma]\subseteq \mathcal{I}_\Sigma\}
$ the singular foliation of vector fields on $W$ tangent to $\Sigma$, where $\mathcal{I}_\Sigma$ stands for the polynomial functions that vanish on $\Sigma$. Here, $W_{\mathrm{reg},\mathfrak F}= W_{\mathrm{reg}}= W\setminus \Sigma$. Consider a geometric resolution $(E_{\bullet}, \dd, \rho)$ of $\mathfrak F$ by trivial vector bundles (which exists because $\mathcal{O}_W$ is Noetherian, see Section 3.3 in \cite{LLS}). %Such an almost Lie algebroid always exists.

Let us show that for every $x\in W\setminus\Sigma$,
$\mathrm{im}(\rho_x)= T_x\mathfrak{F}= T_xW$. It is clear that $\mathrm{im}(\rho_x)\subseteq T_xW$. Conversely, it is a classical property that $x\in W$ is a regular point if and only if there exists “local coordinates” $y_1,\ldots, y_d\in \mathcal O_x $ such that $W$ is of the form
 $$y_1 = \cdots = y_k = 0,$$
 i.e., the localization of $\mathcal I_W $ is generated by these variables, where $\mathcal O_x$ denotes the local ring at $x$. Hence, the tangent space of $W$ at $x$ is the vector space, $\mathrm{span}\left\{{\frac{\partial}{\partial y_i}}_{|_m},\, i\geq k+1\right\}$. Therefore, for $v\in T_xW$ the local vector field
 $$X=\sum_{i=1}^{\dim W}v_i\frac{\partial }{\partial y_{k+i}}$$ maps $ \mathcal{O}_{x}$ to $\mathcal{O}_{x}$, in particular it maps $\mathcal{O}$ to $\mathcal{O}_{x}$ and we have ${X}[\mathcal I_W]\subset ({\mathcal{I}_W})_{\mathfrak m_x}$. Therefore, for every $i\in \{1,\ldots,d\}$, there exists a  polynomial function $g_i$ that does not vanish at $x$ such that $g_iY[x_i]\in\mathbb{C}[x_1,\ldots, x_d]$. By construction, the vector field $\hat{X}=\frac{g_1\cdots g_{r}}{g_1(x)\cdots g_{r}(x)}X$ is tangent to $W$, i.e., $\hat{X}[\mathcal I_W]\subset \mathcal{I}_W$, and satisfies $\hat{X}(x)=v$.
 
 %Notice that for every $x\in \Sigma$, $\mathrm{im}(\rho_x)=\{0\}$. 
 The map $\pi_0\colon W\setminus\Sigma \longrightarrow \mathrm{Gr}_{-(d-r)}(\mathbb{C}^d)\,\; x\longmapsto \mathrm{im}(\rho_x)=T_xW$ is the so-called Gauss map \cite{D.T}. The Zariski closure $\widetilde{W}_0$ of the image of such a map is by definition the classical Nash blowup of $W$ along  its singular locus $\Sigma$.
\end{example}

\begin{example}(\emph{Monoidal transformation}).\label{Monoidal}
Let $W=\mathbb R^d$ or $\mathbb C^d$. Let $\mathcal{I}\subseteq \mathcal{O}_W$ an ideal and let  $C=Z(\mathcal{I})\subset \mathbb C^d$ be the zero locus of the ideal $\mathcal I$. Let $\mathfrak{F}=\mathcal{I}\mathfrak{X}(W)$ the singular foliation of vector fields vanishing along $C$. By Hilbert’s Syzygy theorem \cite{EisenbudSyzygies},
there exists a free resolution {of finite length} for the ideal $\mathcal I$ of polynomial functions  vanishing on $C$  of the form\begin{equation}\label{label:vanishing1}
    \cdots \longrightarrow K_{-2}\stackrel{\partial}{\longrightarrow}K_{-1}\stackrel{\partial}{\longrightarrow}\mathcal{I}\longrightarrow 0
\end{equation} Since $\mathfrak{X}(W)$ is a flat $\mathcal O_W=\mathbb C[x_1, \ldots, x_d]$-module (in fact $\mathfrak X(W)\simeq \mathcal O_W^d$ is a free module),  the sequence \begin{equation}\label{label:vanishing}
\xymatrix{\cdots\ar[rr]^-{\dd=\partial\otimes\text{id}} & & K_{-2} \otimes_{\mathcal O_W} {\mathfrak{X}(W)} \ar^{\dd=\partial\otimes \text{id}}[rr]& & K_{-1}\otimes_{\mathcal O_W}\mathfrak{X}(W)\ar^-\rho[rr] & & \mathfrak{F}.
}\end{equation} is a free resolution $\mathbb{K}\mathbb[W]$ by finitely generated $\mathbb{K}\mathbb[W]$-modules of the singular foliation $\mathfrak{F}=\mathcal I\mathfrak{X}(W)$, where for $(\mu_1, \ldots, \mu_k)$ a set of generators of $K_{-1}$ the anchor map is given by, $\rho(\mu_i\otimes \frac{\partial}{\partial y_j})= \partial (\mu_i)\frac{\partial}{\partial y_j}$, for $i=1,\ldots,k$ and $j=1, \ldots, d$. By Theorem 2.1 in \cite{CLRL},  $\mathfrak{F}$ admits a universal Lie $\infty$-algebroid structure over the complex \eqref{label:vanishing} whose unary bracket is $\ell_1=\partial\otimes \mathrm{id}$ and whose anchor is $\rho$.

Here, $W_{\mathrm{reg},\mathfrak{F}}=W\setminus C$. For $i=1,\ldots,k$, let  $f_i:=\partial (\mu_i)\in \mathcal{I}$.  A direct computation shows that, for every $x\in W\setminus C$, $\ker \rho_x$ is equal to $d$ copies of  $[f_1(x):\cdots :f_k(x)]^\perp$, i.e., $$\ker \rho_x=\left([f_1(x):\cdots :f_k(x)]^\perp\right)^d,$$ where $[f_1(x):\cdots :f_k(x)]$ is a well-defined straight line  of $\mathbb{K}^{k}$ generated by the vector $(f_1(x),\ldots,f_k(x))\in \mathbb{K}^{k}$  seen as a point of the projective space $\mathbb P^{k-1}(\mathbb C)=\mathrm{Gr}_{-(k-1)}(\mathbb{C}^k)$.

 One has, $$\pi_1^{-1}(x)=\begin{cases} \left([f_1(x):\cdots :f_k(x)]^\perp\right)^d,\;
    \; \text{for $x\in W\setminus C$,}\\\\ V^d\in\left(\mathrm{Gr}_{-1}(\mathbb{C}^k)\right)^d\;\text{such that}\; \exists\, (x_n)\in W_{\mathrm{reg},\mathfrak{F}}^{\mathbb{N}},\,\; [f_1(x_n):\cdots :f_k(x_n)]^\perp\underset{n \to +\infty}{\longrightarrow} V,\\ \text{ with}\; V\in\mathrm{Gr}_{-1}(\mathbb C^k)\; \text{as}\; x_n\underset{n \to +\infty}{\longrightarrow}x, \;\text{for $x\in C$}.\end{cases}$$
The $d$ components converge if and only if one of them converges.
    Since $[f_1(x_n):\cdots :f_k(x_n)]^\perp$ converges in $\mathrm{Gr}_{-1}(\mathbb K^k)$ if and only if the straight line $[f_1(x_n):\cdots :f_k(x_n)]$  converges in $\mathbb P^{k-1}(\mathbb K)$, $\widetilde{W}_1$ corresponds to the usual monoidal transformation of $W$ with respect to $\mathcal{I}$ (see for instance \cite{HauserHerwig} or Section \ref{sec:smooth}). In particular, $\widetilde{W}_1$ does not depend, up to isomorphism over $W$, on the choice of the generators $f_1,\ldots, f_k$.

    When $f_1,\ldots, f_k$ form a regular sequence, let us prove that for each $i\geq 1$, $\widetilde W_i$ is again the blowup of $\mathbb C^d$ along $\mathcal I$. The complex in Equation \eqref{label:vanishing1} can then be chosen to be the Koszul complex. Its dual complex is given by the differential map $$\mathfrak X^\bullet(\mathbb C^k)\stackrel{\partial^*}{\longrightarrow}\mathfrak X^{\bullet+1}(\mathbb C^k),\, \frac{\partial}{\partial x_1}\wedge\cdots\wedge\frac{\partial}{\partial x_p}\longmapsto \frac{\partial}{\partial x_1}\wedge\cdots\wedge\frac{\partial}{\partial x_p}\wedge U $$where $U=\sum_{\lambda=1}^k f_\lambda\frac{\partial}{\partial x_\lambda}$. For a sequence of regular points $(x_n)$ that converges to $x\in C$, it is easily checked that  $\mathrm{im}(\dd^*_{x_n})$ converges if and only if $[f_1(x_n): \cdots :f_k(x_n)]$ does in the projective space. This proves the result. As a consequence, $\widetilde W_\infty$ is also the blowup of $W= \mathbb C^d$ along $\mathcal{I}$.
    \end{example}

%section{Conclusion}The construction by O. Mohsen \cite{MohsenOmar} belongs to a large class of monoidal transformations. We have also demonstrated this in a constructive manner. Additionally, we have shown that it is possible to construct other blowup spaces where the singular foliation pulls back well in a similar fashion.% We should then understand $\widetilde W_0$ as the Nash modification of $W$ \cite{D.T}. The meaning of $\widetilde W_1$, when $\mathfrak{F}$ is the vector fields that vanish along an affine subvariety, is the monoidal transformation of $W$ along of that subvariety. Also, $\widetilde{\mathfrak F}_1$ is projective, and is generated by the restrictions of vector fields that are included into the vector fields tangent to $\widetilde W_1$ and  projectable to $\mathfrak F$. It is logical to apply the construction again to $\widetilde W_1$ or to $\widetilde W_\infty$ apply it to a sub-singular foliation of $\mathfrak F$. 

%Also, we would like to relate the $\widetilde W_i$'s and the $\widetilde {\mathfrak{F}}_i$ together.

%Notice that until now we have used the anchor map, the unary bracket and the 2-ary bracket of a universal Lie-$\infty$-algebroid. It would be beneficial to explore further, particularly to elucidate the role of the 3-ary bracket of the universal Lie $\infty$-algebroid of $\mathfrak F$ in this context.%it would be good would like to go further to understand, e.g., the role of the $3$-ary bracket of the universal Lie $\infty$-algebroid  of $\mathfrak F$ in this procedure.

\appendix
\section{Grassmann bundles}\label{appendix}
For $E$ a finite dimension vector space over a field $\mathbb{K}\in \{\mathbb{R}, \mathbb{C}\}$, we denote by 
 $\mathrm{Gr}_{-r}(E)$ the set of all  vector subspaces of $E$ of co-dimension $r\in \mathbb N$. Let us recall a few facts on $\mathrm{Gr}_{-r}(E)$.
 \subsection{Topological structure} $\mathrm{Gr}_{-r}(E)$ is metric space, the corresponding metric is defined by \begin{equation}\label{eq:gap}
     {\displaystyle \delta(V,V')=\lVert {P}_{V}-{P}_{V'}\rVert,}
 \end{equation} where  ${P}_V$ stands for the orthogonal projection of $E$ onto $V\subset E$. It is important to notice that: for all $V, V\in\mathrm{Gr}_{-r}(E)$,\;\begin{equation*}
        \delta(V,V')= \delta(V^\perp,{V'}^\perp)
    \end{equation*} here $V^\perp$ stands for the orthogonal space of $V$. It is proven  (see e.g., \cite{FERRER1994229}) that $\mathrm{Gr}_{-r}(E)$ equipped with the topology induced by the so-called “gap” metric \eqref{eq:gap}, is equivalent to the Grassmann topology, i.e., the topology on $\mathrm{Gr}_{-r}(E)$ whose open subsets $\mathcal{W}\subseteq{\mathrm{Gr}_{-r}(E)}$ are such that $\tau^{-1}(\mathcal{W})$ is open in $\mathrm{St}_r(d, \mathbb{K}):=\left\{A\in\mathrm{M}_{d\times r}(\mathbb{K})\,  \middle|\, \mathrm{rk}(A)=r  \right\}$, with $$\tau\colon \mathrm{St}_r(d, \mathbb{K})\longrightarrow\mathrm{Gr}_{-r}(E),\, A\longmapsto \{\text{vector space spanned by the columns of $A$}\}.$$ 
 Also, $\mathrm{Gr}_{-r}(E)$ is a compact space. 

 \subsection{Manifold structure}\label{Manifold-struct}$\mathrm{Gr}_{-r}(E)$ is moreover a compact manifold of dimension $r(d-r)$ and also, a projective variety.
 \begin{enumerate}
     \item [1.]
 \textbf{Coordinates charts:} 
  One manner to define the standard affine coordinates on the Grassmannian $\mathrm{Gr}_{-r}(E)$ is as follows. Fix a basis $e_1,\ldots,e_{d=\dim E}$ for $E$. Let us describe the first chart. Consider
 %$$A\in \mathrm{St}_r(d,\mathbb{K})\stackrel{\text{Elementary operations}}{\longmapsto} \begin{pmatrix}I_{d-r}\\ A'\end{pmatrix}\in \mathrm{St}_r(d,\mathbb{K})$$ where $A'=(a_{ij})$ is a $r\times (d-r)$-matrix. In that case, $\tau\left( \begin{pmatrix}I_{d-r}\\ A'\end{pmatrix}\right)=V$  admits a basis of the form\begin{equation}v_j := e_j+\sum_{k=1}^\ell a_{kj}e_k,\quad j=1,\ldots,d-r.\end{equation}$V$ is completely determined by $A$. One define an atlas on $\mathrm{Gr}_{-r}(E)$ as follows: Consider the map
\begin{align*}
    \psi\colon M_{r,d-r}(\mathbb{K})&\longrightarrow M_{d,d-r}(\mathbb{K})\\A'\;&\longmapsto \begin{pmatrix}
 I_{d-r}\\ A'
 \end{pmatrix}.
\end{align*}
The vector space $V=\tau\left( \begin{pmatrix}I_{d-r}\\ A'\end{pmatrix}\right)$  admits a basis of the form\begin{equation}\label{eq: chart-grass}v_j:= e_j+\sum_{k=1}^\ell a_{kj}e_k,\quad j=1,\ldots,d-r.\end{equation}$V$ is completely determined by the matrix $A'$. Hence, $\tau\circ \psi$ is the first chart.
%For a permutation, %$\sigma\in \mathfrak{S}_d$, we define the map $\psi_{\sigma(1),\ldots,\sigma(d)}$ that associates  every $A\in M_{r,d-r}(\mathbb{K})$ the matrix in $M_{d,d-r}(\mathbb{K})$ that permutes the lines of $\begin{pmatrix}I_{d-r}\\ A\end{pmatrix}$ w.r.t $\sigma$, that is to say $$\psi_{\sigma(1),\ldots,\sigma(d)}(A):=P(\sigma)\circ\psi_{1,\ldots,d}(A)$$let $P(\sigma)$ be the  permutation matrix of lines associated to $\sigma$.

\item[] For a permutation $\sigma\in \mathfrak S_d$, %For every ordered integers ${\displaystyle 1\leq i_{1}<\cdots <i_{d-r}\leq d}$, the coordinates chart 
let $P(\sigma)$ be the  permutation matrix of lines associated to $\sigma$. We claim that the family $\tau\circ P(\sigma)\circ\psi(M_{r,d-r}(\mathbb{K}))%=\mathcal U_{i_{1},\ldots ,i_{d-r}}\subset \mathrm{Gr}_{-r}(E)
$, indexed by $\sigma\in \mathfrak S_d$ is an atlas of $\mathrm{Gr}_{-r}(E)$. Its image consists of \eqref{eq: chart-grass} up to permutation. 

%on $\mathrm{Gr}_{-r}(E)$ is the open set  consisting of all sub-vector space of $E$ such that for every basis the submatrix which is made  of the rows $i_{1},\ldots ,i_{d-r}$ is invertible, and the coordinate map is $\tau\circ\psi_{\sigma(1),\ldots,\sigma(d)}$ with $\sigma$ is a permutation sending $1,\ldots,d-r$ on $i_{1},\ldots ,i_{d-r}$.\\
 
\item [2.]\textbf{Grassmann bundle:} For $E \to M$ a vector bundle of rank $d$ over a manifold $M$ (or a quasi-projective variety \footnote{the intersection inside some projective space of a Zariski-open and a Zariski-closed subset.}). Let $r\leq d$. The disjoint union:
  $$\mathrm{Gr}_{-r}({E}):= \coprod_{x\in M} \mathrm{Gr}_{-r}(E|_{x})  $$
  comes equipped with a natural manifold structure in the smooth or complex case and a quasi-projective variety structure when $M$ is a quasi-projective variety.  Also
 \begin{equation}\label{eq:Pi}
     \Pi\colon \mathrm{Gr}_{-r}(E)\longrightarrow M
 \end{equation}
 is a fibration. It is called \emph{$(d-r)$-th Grassmann bundle}. %To fix notations, for $x\in M$, elements of the fiber $\Pi^{-1}(x)$, i.e.,  points of $\mathrm{Gr}_{-r} ( E|_x )$, are denoted by  $(V,x)$  with $ V\in \mathrm{Gr}_{-r} ( E|_x )$.
 
 For every open subset $\mathcal U\subset M$ on which $E$ is trivial, $\Pi^{-1}(\mathcal U)\simeq \mathcal{U}\times\mathrm{Gr}_{-r}(\mathbb{K}^d).$ An \emph{adapted chart} for $\mathrm{Gr}_{-r}({E})\longrightarrow M$ around a point $x\in M$ is a set of local coordinates of the form $(\Pi^*x_1,\ldots \Pi^*x_n,z_1,\ldots, z_{r(d-r)}),$ where $(x_1,\ldots,x_n)$ are local coordinates on $M$ and $(z_1,\ldots, z_{r(d-r)})$ are functions which are standard affine coordinates on an open subset of each fiber of $\Pi$ as in item (1). 
\begin{convention}
{Let $x\in M$. Let $e_1, \ldots, e_d$ be local frame for $E$ in a neighborhood $\mathcal{U}$ of $x$. For $y\in \mathcal{U}$, let $\kappa_y$ be the linear isomorphism  defined by
\begin{equation*}
    \kappa_y\colon  E_x \longrightarrow E_y,\; \kappa_y (e_i(x)) = e_i(y),\quad\text{for all}\quad i \in \{1,\ldots, d\}.
\end{equation*} Let $(x_n)$ be a sequence of $M$ that converges to $x$. We will say that a sequence of vector space  $V_{x_n}\in \mathrm{Gr}_{-r}({E}) $ with $V_{x_n}\subset E_{x_n}$, converges to $V\subset E_x$ and write $V_{x_n}\underset{n\rightarrow+\infty}{\longrightarrow} V$ if 
\begin{equation*}
  \kappa_{x_n}^{-1}(V_{x_n})\underset{n\rightarrow+\infty}{\longrightarrow} V \; \text{in}\;\mathrm{Gr}_{-r}({E_x}).
\end{equation*}}
In the sequel, we will not mention $\kappa_{x_n}$, since this notion of convergence does not depend on the  chosen local frames of $E$.
\end{convention}
 %considering an open set $x\in U\subset M$ that trivializes $E$  and a local  frame $\xi=(e_1 ,\ldots , e_r )$ of $E$. Hence, induces a diffeomorphism $\theta_x^\xi \colon \mathrm{Gr}_{-r}(\mathbb{K}^d) \longrightarrow \mathrm{Gr}_{-r}(E_{x})$ that depends smoothly on $x$.
 \item[3.] \textbf{Tautological subbundle:} The Grassmann bundle $\mathrm{Gr}_{-r}(E)$ comes equipped with two vector bundles $\tau^E$ and $A^E$, called tautological subbundle and tautological quotient bundle, that fit into the exact sequence\begin{equation}
     0\longrightarrow \tau^E\longrightarrow \Pi^*E\longrightarrow A^E\longrightarrow 0.
 \end{equation}
 %\begin{equation*}\xymatrix{&  &E\ar[d]\\\tau^E\subset\Pi^*E\ar[r]&\mathrm{Gr}_{-r}(E)\ar[r]&M}\end{equation*}
 Precisely, the fiber  of $\tau^E$ over the point $V\in \Pi^{-1}(x)$ is the codimension $r$ subvector space $V$ of $E|_x=E|_{\Pi(V)}=(\Pi^*E)|_{V}$. By construction, $\tau^E$
is a subbundle of the pull-back bundle $\Pi^*E$. Furthermore,  $A^E\simeq \Pi^*E/\tau^E$.
 
 This tautological quotient bundle is important for us to express some results of this paper.
 \end{enumerate}

\noindent
\textbf{Acknowledgements}. The main results of this paper are taken from Chapter 7 of my PhD thesis \cite{LouisRuben2023UhLa}. I would like to thank C. Laurent-Gengoux, my PhD supervisor, for supporting me in writing this article and directing me to such questions that are originated from Claire Debord and Georges Skandalis. I would also like to thank the University of Lorraine for their financial support through an A.T.E.R position (2022-2023). Additionally, I acknowledge the full financial support for this postdoc from Jilin University.
\bibliographystyle{alpha}
\bibliography{biblio}

\end{document}